\documentclass[12pt, a4paper]{siamart250211}
\usepackage{graphicx}
\usepackage{subfig}
\usepackage{amsfonts}
\usepackage{amsmath}
\usepackage{amssymb}
\usepackage{url}
\usepackage{fancyhdr}
\usepackage{indentfirst}
\usepackage{enumerate}
\usepackage{longtable}
\usepackage{float} 
\usepackage{bbm}
\usepackage{todonotes}
\newcommand{\com}[1]{\todo[inline, size=\small, color=yellow!40]{#1}}

\usepackage{dsfont}
\usepackage{color}
\usepackage{natbib}
\usepackage{comment}
\usepackage{capt-of}
\usepackage{multirow}

\def\red#1{\textcolor{red}{#1}}

\def\d{\mathrm{d}}

\renewcommand{\S}{\mathcal S}

\newcommand{\D}{\mathcal {D}}

\newcommand{\X}{\mathcal {X}}

\newcommand{\E}{\mathbb{E}}
\newcommand{\Q}{\mathbb{Q}}

\newcommand{\F}{\mathcal{F}}

\newcommand{\R}{\mathbb{R}}
\newcommand{\N}{\mathbb{N}}
\newcommand{\p}{\mathbb{P}}

\newcommand{\Y}{\mathcal{Y}}
\newcommand{\range}{\text{range}}

\renewcommand{\L}{\mathcal{L}}

\newcommand{\id}{\mathds{1}}

\newcommand{\dom}{\text{\rm dom\,}}
\newcommand{\btheta}{\boldsymbol{\theta}}
\newcommand{\bpi}{\boldsymbol{\pi}}
\renewcommand{\(}{\left(}
\renewcommand{\)}{\right)}
\renewcommand{\[}{\left[}
\renewcommand{\]}{\right]}
\renewcommand{\ge}{\geqslant}
\renewcommand{\le}{\leqslant}
\renewcommand{\geq}{\geqslant}
\renewcommand{\leq}{\leqslant}
\renewcommand{\epsilon}{\varepsilon}

\newtheorem{example}{Example}

\newtheorem{assumption}{Assumption}
\newtheorem{remark}{Remark}

\newcommand{\cet}{\begin{center}}
	\newcommand{\ecet}{\end{center}}

\usepackage{refcount}
\setrefcountdefault{1}

\newcounter{saveexample}

\allowdisplaybreaks[4]

\begin{document}
\sloppy 

\title{Dynamic Lagrange Multipliers in a Non-concave Utility Framework}

\author{Yang Liu\thanks{\scriptsize  
		School of Science and Engineering, The Chinese University of Hong Kong (Shenzhen), 
  China. Email: \texttt{yangliu16@cuhk.edu.cn}}		
  \and
  Alexander Schied\thanks{\scriptsize 
		Department of Statistics and Actuarial Science, University of Waterloo, Canada. Email: \texttt{aschied@uwaterloo.ca}}		
\and
 Zhenyu Shen\thanks{\scriptsize 
		School of Science and Engineering, The Chinese University of Hong Kong (Shenzhen), 
 China. Email: \texttt{zhenyushen@link.cuhk.edu.cn}}		
}

\date{}

\maketitle
\large
\begin{abstract}	
    In continuous-time portfolio selection for non-concave utility functions, the martingale duality approach is widely adopted in complete markets, while the dynamic programming approach may sometimes lead to singular solutions of the Hamilton-Jacobi-Bellman (HJB) equation. We propose ``dynamic Lagrange multipliers" in a non-concave utility framework, bridging two approaches and demonstrating that the Lagrangian multiplier function (in the martingale duality approach) equals the conjugate dual point related to the value function (in dynamic programming), which is exactly its partial derivative with respect to wealth. Moreover, the dynamic multiplier process exhibits homogeneity via the optimal wealth and pricing kernel processes, offering intuitive economic interpretations as a dynamic shadow price of the envelope theorem. Finally, classical optimal results are recovered and numerically validated by non-concave utility examples.
	
	
	\vskip 10pt  
	
	
	\vskip 10pt

\end{abstract}
\section{Introduction}\label{sec:intro}

Utility maximization is a fundamental problem in continuous-time portfolio selection. We begin with formulating the \cite{M1969,M1971} problem: 
\begin{equation}\label{eq:main prob}
\begin{aligned}
\sup_{\boldsymbol{\pi} \in \mathcal{V}[0,T]} \mathbb{E}\left[ U(X_T) | X_0 = x_0\right],
\end{aligned}
\end{equation}
where $T > 0$ is the terminal time horizon and $\{X_t\}_{0 \leq t \leq T}$ is the wealth process governed by $\boldsymbol{\pi} = \{\boldsymbol{\pi}_t\}_{0 \leq t \leq T}$. For a fixed time $t \in [0, T)$, $\mathcal{V}[t,T]$ denotes the set of admissible controls $\boldsymbol{\pi}$ from time $t$ to $T$. Here, $x_0$ represents the initial wealth and $U$ is the utility function. 

Under the standard framework with a strictly concave utility, the seminal works of \cite{M1969,M1971} pioneered the stochastic control approach to utility maximization. This approach formulates the problem within a dynamic programming framework, where solutions are derived through the Hamilton-Jacobi-Bellman (HJB) equation. It yields both the value function ($u(t,x)$, representing the optimal expected utility) and the optimal portfolio strategy ($\boldsymbol{\Pi}(t, X_t^*)$, expressed as a feedback control of the optimal wealth process $\{X_t^*\}_{0 \leq t \leq T}$). The martingale duality approach, developed in \cite{P1986}, \cite{KLS1987}, and \cite{CH1989}, offers an alternative approach and remains valid beyond Markovian market settings. This method utilizes the pricing kernel process $\{\xi_t\}_{0 \leq t \leq T}$ to establish a risk-neutral measure, employing duality analysis and martingale representation to characterize the optimal terminal wealth. The optimal wealth process and the portfolio process $\{\boldsymbol{\pi}_t^*\}_{0 \leq t \leq T}$ are then constructed through replication arguments. 

Financial practice often requires extensions beyond concave utilities. The best-known example are the $S$-shaped utility functions proposed by \cite{Markowitz} and in the prospect theory of \cite{KahnemanTversky}. Another example arises in hedge fund compensation structures. Convex schemes (e.g., call options in \cite{C2000} and first-loss contracts in \cite{HK2018}) are commonly employed to align hedge fund managers' incentives with fund performance, but create non-concavities in objective functions with respect to terminal wealth. To address this, the prevailing methodology applies concavification techniques to transform the utility $U$ into its concave envelope (the smallest concave function larger than or equal to $U$), after which the martingale duality approach is implemented; see \cite{C2000}, \cite{R2013}, \cite{BS2014}, \cite{DZ2020}, and \cite{LL2020}. In contrast, the dynamic programming approach faces analytical challenges due to the singularity in the associated HJB equations in such cases. Overall, these motivations and challenges have spurred extensive research in non-concave utility optimization; see \cite{CHT2019}, \cite{NS2020}, \cite{DKQW2022}, \cite{LLZ2025}, \cite{QY2026}, and \cite{QYZ2026}.

Although both martingale duality and dynamic programming approaches yield identical solutions ($\boldsymbol{\pi}_t^* = \boldsymbol{\Pi}(t, X_t^*)$) in the classical framework, their analytical structures and economic interpretations differ fundamentally. The key distinction lies in their temporal focus: 
While the martingale method solves the portfolio problem in one step by optimizing a static terminal payoff, dynamic programming requires that optimality be verified locally and infinitesimally along the whole time horizon.
In this paper, our aim is to bridge these approaches 
and establish their dynamic relationships.


Our contribution is to demonstrate the relationship between the dynamic Lagrange multiplier function $\Y(t, x)$ in the martingale duality method and the conjugate dual point $\lambda(t, x)$ related to the value function in dynamic programming, under a framework of non-concave utilities. 
We establish two main results: (i) the dynamic Lagrange multiplier function equals the conjugate dual point, which is exactly the partial derivative of the value function with respect to wealth: $\frac{\partial}{\partial x} u(t,x) = \lambda(t, x) = \Y(t, x)$; and (ii) the dynamic Lagrange multiplier process $\lambda(t, X^*_t)$ exhibits homogeneity via the optimal wealth and the pricing kernel: $\lambda(t, X^*_t) = \lambda(0, x_0) \xi_t$. These results can be seen as the dynamic version of the envelope theorem (well-known in economics) related to the Lagrange multiplier (known as the shadow price). 
Furthermore, we derive the optimal portfolio in a feedback form using these dynamic conjugate relationships ($\boldsymbol{\pi}_t^* = \boldsymbol{\Pi}(t, X_t^*)$). Specifically, a key challenge in our duality analysis in the non-concave utility framework is that the function obtained via the Legendre-Fenchel transform (denoted by $I$ in Equation \eqref{eq:I}) fails to be continuous, and consequently, it is not differentiable. To overcome this obstacle, we employ tools from distribution theory (\cite{SS2010}), which allow us to handle such irregularities in a rigorous manner. 


The paper is structured as follows. Section \ref{sec:model} formulates the model setup 
and Section \ref{sec:DMCT} introduces the tools of convex analysis. 
Section \ref{sec:dynamic Lagrange multiplier theory} presents the main theoretical results and 
Section \ref{sec:proof} gives the detailed proof. Section \ref{sec:example} provides a numerical example that highlights the economic intuition. Additional proofs are deferred to the appendix. 

\section{Model Setting and Problem Formulation}\label{sec:model}
Consider the filtered probability space $(\Omega, 
\{\F_t\}_{0\leq t\leq T}, \p)$ satisfying the usual condition, in which $\{\mathbf{W}_t\}_{0\leq t\leq T}$ is a $d$-dimensional standard Brownian motion.
We assume that the market contains $(d+1)$ assets, denoted by $\{S_{i,t}\}_{0\leq t\leq T}$, $i = 0, 1, \ldots, d$. Among these assets, $\{S_{0,t}\}_{0\leq t\leq T}$ represents the risk-free asset (i.e., bond) and the others are risky assets (i.e., stocks) denoted by $\{\mathbf{S}_t\}_{0\leq t\leq T} := \{(S_{1,t},\ldots, S_{d,t} )^\top\}_{0\leq t\leq T}$. Let $\boldsymbol{\mu}\in \R^d $ be a constant vector and $\boldsymbol{\sigma} \in \R^{d\times d}$ be a constant matrix. We further assume that there exists $\epsilon > 0$ such that
\begin{equation*}
\boldsymbol{\boldsymbol{\eta}^\top \(\boldsymbol{\sigma}\boldsymbol{\sigma}^\top \)\boldsymbol{\eta} } \geq \epsilon \| \boldsymbol{\eta}\|_2^2\quad \text{for any }\boldsymbol{\eta} \in \R^d,
\end{equation*}
which implies that $\boldsymbol{\sigma}$ is invertible.
The assets satisfy the Black-Scholes model
\begin{equation*}
\d S_{0,t} = r S_{0,t} \d t,\quad
\d S_{i,t} = \mu_i S_{i,t} \d t + \boldsymbol{\sigma}_i^\top S_{i,t} \d \mathbf{W}_t, \quad i = 1,\ldots, d,
\end{equation*}
where $\boldsymbol{\sigma}_i$ denotes the $i$-th row of $\boldsymbol{\sigma}$. Let $\mathbf{1}_d = (1, \ldots, 1)^\top \in \R^d$. 
Define the vector $\boldsymbol{\theta} := \boldsymbol{\sigma}^{-1}(\boldsymbol{\mu} - r \mathbf{1}_{d} ) $ and the pricing kernel $\{\xi_{t}\}_{0\leq t\leq T}$ by
\begin{equation*}
    \xi_t = \exp \left\{-\(r+\frac{\|\boldsymbol{\theta}\|_2^2}{2}\)t-\boldsymbol{\theta}^\top \mathbf{W}_t \right\},
\end{equation*}
where $\|\boldsymbol{\theta}\|_2 = \(\sum_{i=1}^d |\theta_{i}|^2 \)^{\frac12}$ is the Euclidean norm of the vector. Finally, \cite{HP1981} and Chapter 18.5 of \cite{S2004} show that such a market is complete, i.e., every contingent claim can be hedged. In fact, all the results in this paper can be generalized to the case where the market coefficients are deterministic processes. Still, we let them be constant numbers/vectors for notational simplicity. 

Let $\boldsymbol{\pi} = \{(\pi_{1,t},\ldots, \pi_{d,t})^\top\}_{0\leq t\leq T}$ denote the portfolio process, where the $i$-th component represents the money that the investor allocates to the $i$-th risky asset. A portfolio $\boldsymbol{\pi} = \{ \boldsymbol{\pi}_t \}_{0 \leq t \leq T}$ is called admissible if $\boldsymbol{\pi}$ is an $\{ \mathcal{F}_t \}_{0 \leq t \leq T}$ -progressively measurable $\R^d$-valued process and $\int_0^T \|\boldsymbol{\pi}_t\|_2^2 \, \d t < \infty$ almost surely. Let $\{X_t\}_{0\leq t\leq T}$ denote the investor's wealth process, whose dynamic is characterized by the following stochastic differential equation (SDE):
\begin{equation}\label{eq:SDE of wealth}
\left\{
\begin{aligned}
\d X_t &=\(r\(X_t - \sum_{i = 1}^d \pi_{i,t}\) +  \sum_{i=1}^d \pi_{i,t}  \mu_i \) \d t +\sum_{i=1}^d \pi_{i,t} \boldsymbol{\sigma}_{i} ^\top \d \mathbf{W}_t,\quad 0\leq t\leq T,  \\
X_0 &= x_0.
\end{aligned}
\right.
\end{equation}
\red{}

The investor is assumed to have a utility function $U:\R \rightarrow \R \cup \{-\infty\}$. Define the effective domain $\dom U := \{x \in \R | U(x) > -\infty\}$. We further propose the following assumption on the utility function: 
\begin{assumption}\label{assp:utility}
The utility function $U: \R \rightarrow \R\cup\{-\infty\}$ has the following properties:
\begin{enumerate}[(i)]
    \item There exists an interval $\dom U \subset \R$ with its left endpoint $L$ being finite and right endpoint being $\infty$, such that $U(x) = -\infty$ for all $x \notin \dom U$ and $U(x)$ is finite otherwise. Hence, $\dom U = [L,\infty)$ or $(L,\infty)$.
    \item $U$ is differentiable except at countably many points and 
        $\lim_{x\uparrow \infty} U'(x) = 0$.
    \item $U$ is increasing (i.e., non-decreasing) and $\lim_{x\uparrow\infty}U(x) = \infty$.
    \item $U$ is upper semicontinuous.
    \item There exist $A, B \in \R$ such that $U(x) < Ax +B $ for all $x \in \dom U$.
\end{enumerate}
\end{assumption}
It is worth noting that the utility function can be discontinuous and non-concave over the entire domain. Moreover, these conditions guarantee the existence of the concave envelope of $U$,
which is the smallest concave function greater than or equal to $U$; see Equation \eqref{eq:Legendre-Fenchel transform} later.
Under the settings above, we investigate Problem \eqref{eq:main prob}.
A wealth process $\{X_t\}_{0\leq t\leq T}$ is called admissible if it is driven by an admissible control and $X_t \in\dom U$ for all $t \in [0,T]$. We define the set of all admissible terminal wealth variables as
\begin{equation*}
    \X_T := \{ X_T \mid \{X_t\}_{0\leq t\leq T} \text{ is an admissible wealth process}\}.
\end{equation*}

\section{Preliminary Analysis and Methodological Foundations
}
\label{sec:DMCT}
To reveal the connection between the martingale duality and dynamic programming methods, we apply each separately to solve Problem \eqref{eq:main prob}; see Chapters 3-4 in \cite{KS1998}.
Under our market setting, Problem \eqref{eq:main prob} is equivalent to the terminal wealth optimization Problem \eqref{prob-kernel}: 
\begin{equation}\label{prob-kernel}
    \sup\Big\{\mathbb E[U(X)] \mid \text{$X$ is $\mathcal F_T$-measurable, $X\in \dom U$ $\mathbb P$-a.s., and $\mathbb E_{\mathbb Q}[e^{-rT}X]\le x_0$}\Big\},
\end{equation}
where $\Q$ is the unique equivalent martingale measure defined by $\frac{\d \Q}{\d \p} = e^{rT} \xi_T$.  

We apply the Legendre-Fenchel transform:
\begin{equation}\label{eq:Legendre-Fenchel transform}
\begin{aligned}
	V(y) &:= \sup_{x \in \text{dom } U} \{ U(x) - xy\},&& y > 0,\\
	U^{**}(x) &:= \inf_{y \in (0, \infty)} \{ V(y) + x y\}, && x \in \R, \\ 
\end{aligned}
\end{equation}
with the convention $\inf \emptyset = \infty$ and $\sup\emptyset = -\infty $.
Here, $V$ is usually called the conjugate function.
Remarkably, $U^{**}$ is the concave envelope of $U$; see, e.g., Section 26 of \cite{R1970} for more details on technical treatments.  In the case where $U$ is concave, we have $U^{**}(x) = U(x)$ for all $x \in \dom U$. Properties of $U^{**}$ are discussed in Lemma \ref{lem:enve properties}, some of which are reflected in the literature on convex analysis (e.g., \cite{BC1994}). 

\begin{lemma}\label{lem:enve properties}
Suppose that Assumption \ref{assp:utility} holds. Then the function $U^{**}$ satisfies:
    \begin{enumerate}[{\rm (a)}]
        \item $U^{**}$ is strictly increasing.
        \item The effective domain $\text{\rm dom } U^{**}$ is equal to $\text{\rm dom }U$.
        \item The set $S:=\{x\mid U^{**}(x)>U(x)\}$ is open. In particular, $S$ can be written as the disjoint union of countably many open intervals $(a_k,b_k)$, $k=1,2,\dots$
        \item $U^{**}$ is affine on each of the intervals $(a_k,b_k)$ from part (c). 
    \end{enumerate}
\end{lemma}

Despite the properties shown in Lemma \ref{lem:enve properties}, we present the following assumption on the concave envelope, which we assume to hold throughout the paper. 
\begin{assumption}\label{assumoption 1}
    The concave envelope is assumed to satisfy the following condition:
    \begin{equation*}
       \lim_{x\uparrow \infty} (U^{**})'(x) = 0.
    \end{equation*}
    Moreover, if $\dom U$ is open, then we additionally assume
    \begin{equation*}
        \lim_{x\downarrow L} \(U^{**} \)'(x) = \infty.
    \end{equation*}
\end{assumption}
Assumption \ref{assumoption 1} guarantees the attainability of Problem \eqref{eq:main prob}. Specifically, it ensures that the function $I$ defined later in Lemma \ref{lem:form of I} is finite. An example utility satisfying Assumption \ref{assp:utility} but violating Assumption \ref{assumoption 1} is shown as follows:
\begin{equation*}
   U(x)=\begin{cases}
\sum_{n=1}^\infty(n+(x-n)/n)\mathbbmss{1}_{[n,n+1)}(x), &\text{$x\ge 1$};\\
-\infty, &\text{otherwise.}
\end{cases}
\end{equation*}
If $\dom U$ is open, this is similar to the Inada condition where $\lim_{x\rightarrow \infty}U'(x) = 0 $; see \cite{KLS1987}. 
Now, we present a semi-analytical function that attains the maximum in \eqref{eq:Legendre-Fenchel transform}.
\begin{lemma}\label{lem:form of I} Suppose that Assumptions \ref{assp:utility}-\ref{assumoption 1} hold.
    Consider the function $I$ defined by
    \begin{equation}\label{eq:I}
        I(y) := \inf\{x \in \dom U: (U^{**})'(x) \leq y \},\quad y > 0,
    \end{equation}
    where $(U^{**})'$ is the right-hand derivative of $U^{**}$. Then for all $y > 0$, we have $I(y) \in \dom U$ and \begin{equation}\label{eq:I is maximizer} U (I (y)) - y I (y) = V (y) < \infty. 
    \end{equation}
\end{lemma}

Note that $I$ is decreasing since $(U^{**})'$ is decreasing.
We further define the function
\begin{equation*}
g\(t,y\) := \E \[ Z_{t,T} I\(y Z_{t,T}\) \],
\end{equation*}
where $Z_{t,T} := \xi_{T}/\xi_{t}$.
Then $g(t,y)$ is strictly decreasing in $y$ since $I$ is 
decreasing and non-constant and the random variable $Z_{t,T}$ has full support on $(0,\infty)$. 

For fixed $t \in [0,T)$, define the Lagrange multiplier function
\begin{equation}\label{eq:def of Y}
    \Y(t,x) := \inf\left\{y\in (0,\infty): g(t,y) \leq x  \right\}.
\end{equation}
It is clear that for $x \in \{z\in \dom U\mid \exists \text{ } y >0,\; g(t,y) = z\}$, $\Y(t,x)$ is the unique number satisfying $x = g(t, \Y(t,x))$. However, this does not hold for those $x \in \dom U$ that are not in the range of $g(t,y)$. For fixed $t \in [0,T)$, the continuity of $g(t,y)$ in $y$ (shown later in Section \ref{sec:technical discussion}) ensures that its range is an interval. To determine the interval, we consider two limits: $\lim_{y\downarrow 0} g(t,y)$ and $\lim_{y \uparrow \infty}g(t,y)$. From Assumption \ref{assumoption 1}, we see that $\lim_{y\downarrow 0} I(y) = \infty$, and hence $\lim_{y\downarrow 0} g(t,y) = \infty$. For the upper limit, we compute
\begin{align}
\lim_{y\uparrow \infty} g(t,y) &= \lim_{y\uparrow \infty}\E\[Z_{t,T}I\(yZ_{t,T}\) \] 
= \E\[\lim_{y\uparrow \infty}Z_{t,T} I(yZ_{t,T})\] 
= L\cdot\E\[Z_{t,T} \] 
= Le^{-r(T-t)},\nonumber
\end{align}
where the second equality comes from the monotone convergence theorem, which holds when $\E[Z_{t,T}I(yZ_{t,T})] < \infty$ for some $y > 0$, $Z_{t,T}I(yZ_{t,T})$ is decreasing in $y$, and $\lim_{y\uparrow\infty}Z_{t,T}I(yZ_{t,T})=L Z_{t,T}$ $\mathbb P$-almost surely. This finiteness of $\E \[ Z_{t,T} I\(y Z_{t,T}\) \]$ is discussed later in Section \ref{sec:technical discussion}. Now, we can define a suitable domain:
\begin{equation}\label{eq:domain_new}
\begin{aligned}
    \D_U &:= (\hat{L}, \infty) \subset \dom U, \;\; 
    \text{where } \hat{L} 
    := \max\left\{ Le^{-rT}, L \right\}.
\end{aligned}
\end{equation}
Restricting the domain of $U$ affects the range of $\Y(t,x)$. We define the following set to denote the range of $\Y(t,x)$ when restricting $x$ to $(\hat{L},\infty )$:
\begin{equation*}
    \D_I := \(0, \lim_{x\downarrow \hat{L}}\Y\(0,x\) \).
\end{equation*}
We see that if $L \geq 0$, Assumption \ref{assumoption 1} and Equation \eqref{eq:domain_new} guarantee that $\D_I = \(0,\infty\)$, aligning with the setting in the standard literature (e.g., \cite{KLS1987}).
Next, to apply the distribution theory later in the proofs, we suppose the following assumption to hold throughout the paper.
\begin{assumption}\label{assp}
    There exist $\delta > 0$, $C_0 > 0$, and $M \in \N_+$ such that
    \begin{equation*}
    \begin{aligned}
        I(y) &\leq C_0 y^{-M}, && y \in (0,\delta),\\
        |U^{**}(x)| &\leq C_0 |x-L|^{-M},&& x \in (L,L+\delta).
    \end{aligned}
    \end{equation*}
\end{assumption}
Examples of piecewise hyperbolic absolute risk averse (PHARA) utilities (i.e., piecewise power/log and exponential utilities satisfy Assumption \ref{assp}; see \cite{LLMV2024}). More details about Assumption \ref{assp} will be discussed in Section \ref{sec:technical discussion}.

Now, we introduce a family of optimal control Problems ($P_{tx}$) parametrized by $(t,x)\in [0,T)\times\D_U$:
\begin{equation*}
	\textbf{Problem ($P_{tx}$)}  \quad	\sup_{X_T \in \X_T}  \E[U(X_T)|X_t = x] = \sup_{X_T \in \X_T, X_t = x}  \E[U(X_T)].
\end{equation*}
We see that Problem \eqref{prob-kernel} is equivalent to Problem ($P_{tx}$) with $t = 0$ and $X_0 = x_0$. 
We proceed to solve Problem ($P_{tx}$). 
For any fixed $t \in [0,T)$, define the probability measure $\Q_t$ by $\d \Q_t / \d \p = e^{r(T-t)}Z_{t,T}$. Such $\Q_t$ exists since $\E[e^{r(T-t)}Z_{t,T}] = 1$.
Applying Itô's lemma and Girsanov's theorem to a discounted wealth process $\{e^{-r(s-t)}X_s\}_{t\leq s\leq T}$, we see that $\{e^{-r(s-t)}X_s\}_{t\leq s\leq T}$ is a continuous local martingale under $\Q_t$ with initial value $X_t = x$. Further, $\{e^{-r(s-t)}X_s\}_{t\leq s\leq T}$ is a $\Q_t$-supermartingale since it is lower-bounded. Hence, we have $\E[Z_{t,T}X_T|\F_t] = \E_{\Q_t}[e^{-r(T-t)}X_T|\F_t] \leq X_t = x$. 

Now, let $U$ be a utility function and $\{X^{t, x}_s\}_{0\leq s\leq T}$ be any wealth process with 
$X_t=x$ almost surely for some fixed $(t,x) \in [0,T)\times \D_U$. Utilizing the convex conjugate of $U$ and Lemma \ref{lem:form of I}, we get
\begin{equation}\label{eq:convex conjugate}
\begin{aligned}
    V(\kappa) &= \sup_{x \in \D_U} \left\{U(x) - \kappa x \right\}
    = U(I\(\kappa\)) - \kappa I\(\kappa\),\quad \kappa >0.
\end{aligned}
\end{equation}
Taking $\kappa = \Y\(t,x\)Z_{t,T}$, Equation \eqref{eq:convex conjugate} shows that
\begin{equation}\label{eq:convex conjugate inequality}
    U \(I\(\Y(t,x)Z_{t,T} \) \) - \Y\(t,x\)Z_{t,T} I\(\Y\(t,x\)Z_{t,T} \) \geq U\(X_T\) - \Y\(t, x\) Z_{t,T} X^{t, x}_T. 
\end{equation}
By the definition of $\Y(\cdot, \cdot)$ and continuity of $g(t,y)$, we know 
$\E\[Z_{t,T} I\(\Y(t, x) Z_{t,T} \)\] = x$. 
Taking the 
expectation on both sides of \eqref{eq:convex conjugate inequality}, we get
\begin{equation*}
   \E\[U \( I\(\Y(t,x) Z_{t,T} \) \)\] \geq \E\[U\(X^{t, x}_T\)\], \;\; \text{for any wealth process $\{X^{t, x}_s\}_{0\leq s\leq T}$.}
\end{equation*}
Hence, we know the optimal terminal wealth for Problem \eqref{prob-kernel} is attainable and given by 
$X_T^{*, t, x} := I\(\Y(t,X_t)Z_{t,T} \) = I\(\Y(t,x)Z_{t,T} \)$. 
We define
\begin{equation*}
\begin{aligned}
        X_s^{*, t, x} &:= Z_{t,s}^{-1} \E\[Z_{t,T} I\(\Y(t,x)Z_{t,T} \)|\F_s \]\\
        &= \E\[Z_{s,T} I\(\Y(t,x)Z_{t,s} Z_{s,T} \) | Z_{t,s}\],\quad s \in[t,T).
\end{aligned}
\end{equation*}
Further, 
if $X_s^{*, t, x}$ is twice differentiable in $Z_{t,s}$ and the following $\boldsymbol{\pi}^*$ satisfies the admissible conditions,
we can apply Itô's formula to $X_s = X_s^{*, t, x}(Z_{t,s})$ to obtain the optimal portfolio of Problem \eqref{eq:main prob}:
\begin{equation}\label{eq:classic_port}
    \boldsymbol{\pi}_s^* := -\boldsymbol{\sigma}^{-1} \boldsymbol{\theta} Z_{t,s} \frac{\partial X_s^{*, t, x} (Z_{t,s})}{\partial Z_{t,s}}.
\end{equation}

Additionally, in the dynamic programming approach, we define the value function for Problem \eqref{prob-kernel} as follows:
\begin{equation*}
    u(t,x) 
    := \sup_{X_T\in \X_T} \E[U(X_T^{t, x})],\quad (t,x) \in [0,T)\times \D_U.
\end{equation*}
Denote by $\{X_s^{t, x, \boldsymbol{\pi^*}}\}_{t\leq s\leq T} $ the wealth process starting at $X_t=x$ and controlled by the above $\boldsymbol{\pi}^*$.
With the discussions above, we have shown that 
\begin{equation}\label{eq:form of value function}
\begin{aligned}
    u(t,x) &= \E\[U(X^{t, x, \boldsymbol{\pi}^*}_T)
    \]\\ 
    &= \E\[U\(I\(\Y(t,X_t^{t, x, \boldsymbol{\pi}^*})Z_{t,T} \) \)
    \]\\
    &= \E\[U^{**}\(I\(\Y(t,X_t^{t, x, \boldsymbol{\pi}^*})Z_{t,T} \)\)
    \] = \E\[U^{**}\(I\(\Y(t, x)Z_{t,T} \)\)
    \],
\end{aligned}
\end{equation}
where the third equality comes 
from Lemmas \ref{lem:enve properties}-\ref{lem:form of I} that $I(y) \notin S$ for all $y >0$ and hence that $U\(I\(y\)\) = U^{**}\(I\(y\) \) $ for all $y >0$.
\begin{remark}\label{rmk:finiteness}
    At this stage, the finiteness of $\E[Z_{t,T}I\(yZ_{t,T} \) ]$ and $\E\[U^{**}\(I\(\Y(t,x)Z_{t,T} \) \) \] $ is not guaranteed, which will be proved in Section \ref{sec:technical discussion}.
\end{remark}

\section{Main Results}
\label{sec:dynamic Lagrange multiplier theory}
We first introduce the dynamic conjugate functions and reveal the connections between the martingale duality approach and the dynamic programming approach through these functions. Remarkably, we will show later in Lemma \ref{lem:concavity of value function} that $u(t,x)$ is strictly concave in $x$ for all $t \in [0,T)$ and hence that the functions in Definition \ref{def:conjugate functions} are well-defined.

\begin{definition}\label{def:conjugate functions}
	For any $t\in[0,T)$ and $y \in \D_I$, we define the conjugate function $v(t, y)$ and the conjugate point $i(t, y)$ as
	\begin{equation}\label{def V}
		\left\{
		\begin{aligned}
			&v(t, y) =  \sup_{x \in \D_U}  \left\{ u(t,x)-x y\right\},\\
			&i(t, y)=	\arg \max_{x \in \D_U}  \left\{ u(t,x)- x y\right\}.
		\end{aligned}
		\right.
	\end{equation}
	The conjugate dual point of the function $v(t, y)$ is
	\begin{equation}\label{def nu_t}
		\lambda(t, x) = \arg\min_{y>0}\{v(t,y)+xy\},\quad 0\leq t< T,\quad x\in\D_U.
	\end{equation}
    Moreover, we define $v(T,y) = V(y)$, $i(T,y) = I(y)$, and $\lambda(T,x) = \Lambda(x)$.
\end{definition}
The conjugate function and conjugate point in \eqref{def V} build up a family of functions over $[0,T]$. One can notice that $v(T,y) = V(y)$ and $i(T,y)=I(y)$. Moreover, as the functions are defined through the Legendre-Fenchel transform of the value function $u(t,x)$, it is natural to investigate their relations to the Legendre-Fenchel transform of the value function at terminal time $T$. The following theorem illustrates such a relationship.
\begin{theorem}\label{prop of V}
Suppose that Assumptions \ref{assp:utility}-\ref{assp} hold.
The conjugate function $v(t,y)$ and the conjugate point $i(t,y)$ have the following consistency properties:
		for any $t\in[0, T)$ and any $y\in\D_I$, 
		\begin{equation*}
            \begin{aligned}
		v(t,y) &= \E \left[ V(y Z_{t, T})\right] = \E\[U\(I\(yZ_{t,T}\)\) - yZ_{t,T}I\(yZ_{t,T}\) \],\\
            i(t, y) &= \E\left[Z_{t, T} I(y Z_{t, T})\right] = g(t,y).
            \end{aligned}
		\end{equation*}
        
\end{theorem}

Now, we present the key equivalence between the dynamic Lagrange multiplier $\Y(t,x)$ and the conjugate dual point $\lambda(t,x)$. Recall that $\Y(t,x)$ is defined as the inverse of $g(t,y) = \E[Z_{t,T}I(yZ_{t,T})]$ in $y$ while $\lambda(t,x)$ comes from the Legendre-Fenchel transform of the value function $u(t,x)$. The following theorems show the implications between the martingale duality approach and the dynamic programming approach by revealing the connections among $\Y(t,x)$, $\lambda(t,x)$, the value function $u(t,x)$, and the pricing kernel $\{\xi_t\}_{0\leq t\leq T}$.
\begin{theorem}\label{thm:Lagrange multiplier identity and homogeneity}
Suppose Assumptions \ref{assp:utility}-\ref{assp} hold.
\begin{enumerate}[(i)]
    \item For any $t \in [0,T)$ and $x \in \D_U$,
    \begin{equation}\label{eq:first_equal}
        \Y(t,x) = \lambda(t,x) = \frac{\partial u(t,x)}{\partial x}.
    \end{equation}
    \item Let $\{X_t^*\}_{0\leq t\leq T}$ be the optimal wealth process of Problem \eqref{eq:main prob} with $X_0^* = x_0$. Then the following homogeneity relation holds: for any $t \in [0,T)$ and $x_0 \in \D_U$, 
\begin{equation*}
    \lambda\(t,X_t^*\) = \lambda\(0,x_0\)\xi_t,\quad \text{almost surely}.
\end{equation*}
\end{enumerate}
\end{theorem}

Theorem \ref{thm:Lagrange multiplier identity and homogeneity} is the main result of our dynamic Lagrange multiplier theory for a non-concave utility. It shows that for a given wealth value $x$ at time $t$, the Lagrange multiplier to Problem ($P_{tx}$) is exactly $\lambda(t, x)$. Moreover, $\Y(t,x)$ and $\lambda(t,x)$ both equal the marginal value function.
In economics, the Lagrange multiplier corresponding to the utility maximization problem is called the \textit{shadow price}, which reflects the marginal benefit of budgets. 
In our case, $\lambda(t,x)$ represents the marginal optimal expected utility at wealth level $x$ and time $t$. 

On the other hand, for fixed $t\in[0,T)$ and $x\in \D_U$, the value function is given by $\E[U(I(\lambda(0,x)\xi_T))|\F_t] = \E[U(I(\lambda(t,X_t^*)Z_{t,T}))|X_t^*]$, which decreases as the random variable $\lambda(t,X_t^*)$ increases. Observing the second part of Theorem \ref{thm:Lagrange multiplier identity and homogeneity} that $\lambda(t,X_t^*) = \lambda(0,x)\xi_t$, we can interpret this as follows: As the value of $\xi_t$ represents the market scenario in the Black-Scholes model (a large $\xi_t$ indicates a bad market performance and vice versa), and from the second part of Theorem \ref{thm:Lagrange multiplier identity and homogeneity}, the increase of $\lambda(t,X_t^*)$ is essentially caused by a poor market performance. 

Finally, we provide the optimal results in the non-concave utility framework with our proposed dynamic Lagrange multiplier, which are different from those in the literature (e.g., \cite{KLS1987} or Equation \eqref{eq:classic_port}). 
\begin{theorem}
	\label{thm-opt-value}
    Suppose that Assumptions \ref{assp:utility}-\ref{assp} hold.
		For Problem \eqref{eq:main prob}, the optimal terminal wealth is given by
		$
		X^*_T = I(\lambda(0, x_0)\xi_T).
		$
        Furthermore, if $X_t^*$ is twice differentiable in $\xi_t$ for all $t \in [0,T)$, then the optimal wealth process is given by
		\begin{equation*}
		X^*_t = i(t, \lambda(0, x_0)\xi_t) = i(t, \lambda(t, X_t^*)),\quad t\in[0,T),
		\end{equation*}
		and the optimal portfolio is
		\begin{equation*}
				\boldsymbol{\pi}_t^* = -\boldsymbol{\sigma}^{-1}\boldsymbol{\theta}\lambda(t, X^*_t) \left.\frac{\partial i(t,y)}{\partial y}\right|_{y = \lambda(t, X^*_t)}, \quad t\in[0,T).
		\end{equation*}
		
		%
		%
\end{theorem}
This theorem shows that the optimal wealth at time $t$ is exactly the conjugate dual point of $\lambda(0, x_0) \xi_t$ corresponding to the function $u(t,\cdot)$.
Moreover, the optimal portfolio $\{\boldsymbol{\pi}_t^*\}_{0\leq t\leq T}$ is a feedback control of the optimal wealth process $X_t^*$, which satisfies $\boldsymbol{\pi}_t^* = \boldsymbol{\Pi}(t, X_t^*)$ and $\boldsymbol{\Pi}$ is a deterministic function
\begin{equation*}
\boldsymbol{\Pi}: [0, T] \times \R \to \R^d, \; (t, x) \mapsto -\boldsymbol{\sigma}^{-1}\boldsymbol{\theta}\lambda(t, x)  \left.\frac{\partial i(t,y)}{\partial y}\right|_{y=\lambda(t, x)} .
\end{equation*}

\section{Main Proofs}\label{sec:proof}

\subsection{Technical Discussion}\label{sec:technical discussion}
In this subsection, we introduce the core concepts supporting our proofs: distribution theory and distributional derivatives (\cite{SS2010}). 
In theory, the Schwartz space on $\R$, denoted by $\S(\R)$, is a space of test functions. For $\varphi \in \S(\R)$, $\varphi$ has derivatives of all orders and satisfies
\begin{equation*}
    \sup_{z \in \R, \alpha, \beta \leq N} |z^\alpha \varphi^{(\beta)}(z) | < \infty,\quad\forall N \in \N_+.
\end{equation*}
We can define a subset of $\S(\R)$ by
\begin{equation*}
    \S_0(0,\infty) := \left\{\varphi \in \S(\R) : \varphi(z) = 0,\;\forall  z\leq 0 \text{ and }\lim_{z \downarrow 0}\left| z^{-\alpha} \varphi^{(\beta)}(z)\right| = 0, \; \forall \alpha, \beta \in \N_+\right\}.
\end{equation*}
It is easy to check that $\S_0(0,\infty)$ is a closed subspace of $\S(\R)$, and we endow it with the relative topology. The space $\S_0'(0,\infty)$ of all continuous linear functionals on $\S_0(0,\infty)$ is called the space of \emph{Schwartz distributions} or \emph{tempered distributions} on $(0,\infty)$. For instance, consider a locally integrable function $H$ on $(0,\infty)$. If there exists $\alpha_0 \in \N_+$ such that
\begin{equation}\label{eq:distribution condition}
    \lim_{z\downarrow 0} |z^{-\alpha_0}H(z)| = \lim_{z\uparrow \infty} |z^{-\alpha_0} H(z) | = 0,
\end{equation}
then $H$ defines a continuous linear functional $\mathcal{H}$ on $\S_0(0,\infty)$ that is given by 
\begin{equation*}
    \mathcal{H}\(\varphi\) = \langle H, \varphi \rangle := \int_0^\infty H(z) \varphi(z)\, \d z < \infty, \quad \varphi \in \S_0(0,\infty).
\end{equation*}
{Remarkably, the distribution on $\S(\R)$ is called the \textit{tempered distribution}, which is of polynomial growth at both negative and positive infinities. When restricting the domain of test functions to $(0,\infty)$, it is natural that the distribution should be of polynomial growth at the endpoints of the domain, i.e., $0$ and $\infty$. This is the reason we define the subspace $\S_0(0,\infty)$.}

Since $I$ and $U^{**}\circ I$ are decreasing, they are locally integrable on $(0,\infty)$. Under Assumption \ref{assp}, it is clear that both $I$ and $U^{**}\circ I$ satisfy the condition given by \eqref{eq:distribution condition}. On the other hand, we can compute the cumulative density function (CDF) of $Z_{t,T}$ as follows:
\begin{equation*}
\begin{aligned}
    F_{Z_{t,T} }(z) &= \p \left\{ \exp{\left\{-\(r+\frac{\|\btheta\|_2^2}{2}\)\(T-t\)-\|\btheta\|_2 \sqrt{T-t}N \right\}\leq z } \right\}\\
    &= \p\left\{ N \geq -\frac{1}{\|\btheta\|_2\sqrt{T-t}}\(\log\(z\)+\(r+\frac{\|\btheta\|_2^2}{2} \)\(T-t\) \) \right\}\\
    &= 1 -\Phi \(\hat{d}\(t,z\) \),
\end{aligned}
\end{equation*}
where $N$ is a standard normal random variable, $\Phi$ is the standard normal CDF, and $\hat{d}(t,z) := -\frac{1}{\|\btheta\|_2\sqrt{T-t}}\(\log\(z\)+\(r+\|\btheta\|_2^2/2 \)\(T-t\) \)$. 
Then by differentiating $F_{Z_{t,T}}$, we can get the probability density function (PDF) 
\begin{equation*}
    f_{Z_{t,T}}(z) = \frac{1}{z\|\btheta\|_2\sqrt{T-t}}\Phi'\(\hat{d}\(t,z\)\)\text{ for } z > 0 \text{ and $0$ otherwise}.
\end{equation*}
Define $\psi(z) := zf_{Z_{t,T}}(z)$.
It is easy to check that $f_{Z_{t,T}} \in \S_0(0,\infty)$ and $\psi \in \S_0(0,\infty)$. Hence, we have
\begin{equation*}
\begin{aligned}
    g(t,y) &= \frac{1}{y}\left\langle I, \frac{\psi}{y} \right\rangle < \infty, \quad 
    u(t,x)
    = \frac{1}{\Y\(t,x\)}\left\langle U^{**}\circ I, \frac{f_{Z_{t,T}}}{\Y\(t,x\)}\right\rangle < \infty,
\end{aligned}
\end{equation*}
which solve the finiteness problem pointed out in Remark \ref{rmk:finiteness}.
This also shows that $g(t,y)$ is continuous in $y$ and $u(t,x)$ is continuous in $x$.
Notably, Assumption \ref{assp} also plays an important role later in the proof of Lemma \ref{lem:distributional derivative}. Treating the expectations from the perspective of distribution theory, we can apply functional analysis techniques to exchange the order of differentiation and expectation. We will see this point later.
\begin{remark}
    We have shown that we can treat $I$ and $U^{**}\circ I$ as Schwartz distributions, which is a crucial conclusion for later proofs. Here, the finiteness results, $g(t,y) < \infty$ and $\E[U^{**}(I(\Y(t,x)Z_{t,T} ))] < \infty$, are direct corollaries. A simpler way to show the finiteness is by noting that Assumption \ref{assp} gives
    \begin{equation*}
        LZ_{t,T}\leq Z_{t,T}I(yZ_{t,T})\leq C_0y^{-M}Z_{t,T}^{1-M},
    \end{equation*}
    and that $I$ is of polynomial decay at infinity. Then $Z_{t,T}I(yZ_{t,T})$ is bounded by two log-normally distributed random variables and is thus integrable. This is similar for u(t,x).
\end{remark}
\subsection{Proof of Theorem \ref{prop of V}}

Before presenting the proof, we give some lemmas that will be used later.
\begin{lemma}\label{lem:distributional derivative}
Suppose that Assumptions \ref{assp:utility}-\ref{assp} hold.
We have the following identity:
\begin{equation*}
    \frac{\partial}{\partial y} \E\[U^{**}\(I\(yZ_{t,T}\) \) \] = y\frac{\partial g(t,y)}{\partial y},\quad (t,y)\in [0,T)\times \D_I.
\end{equation*}
\end{lemma}
\begin{proof}
Recall the distribution theory discussed in Section \ref{sec:technical discussion}. Let $\mathcal{I}$ be the Schwartz distribution on $\S_0(0,\infty)$ defined by $I$, i.e., 
\begin{equation*}
\mathcal{I}\(\varphi\) := \langle I, \varphi \rangle = \int_0^\infty I(z) \varphi(z)\, d z,\quad \varphi \in \mathcal{S}_0(0,\infty).
\end{equation*}
Then we can define the distributional (weak) derivative of $I$ by $I'$. We want to show that the following equation holds:
\begin{equation}\label{eq:exchange diff and int}
    \frac{\partial }{\partial y} \int_0^\infty I(yz) \varphi(z)\d z = \int_0^\infty I'(yz) z\varphi(z)\, \d z,\quad \varphi \in \mathcal{S}_0(0,\infty).
\end{equation}
To show this, we begin by computing the right-hand side of the equation:
\begin{align}
    \int_0^\infty I'\(yz\)z\varphi(z)\, \d z &= \frac{1}{y} \int_0^\infty I'\(u\) \frac{u}{y} \varphi\(\frac{u}{y}\)\, \d u\nonumber\\
    &= \frac{1}{y^2} \left\langle I', u \mapsto u\varphi\(\frac{u}{y} \)\right\rangle\nonumber\\
    &= -\frac{1}{y^2} \left\langle I, u \mapsto \varphi\(\frac{u}{y}\) + \frac{u}{y} \varphi'\(\frac{u}{y}\) \right\rangle,\nonumber
\end{align}
where the last equality comes from the definition of the distributional derivative. We show some important properties before computing the left-hand side of Equation \eqref{eq:exchange diff and int}. For $y \in \D_I$, define $\varphi_y(z) := \varphi(z/y)$. Then $\varphi_y \in \mathcal{S}_0(0,\infty)$ if $\varphi \in \mathcal{S}_0(0,\infty)$. Moreover, we have $\frac{\partial \varphi_y}{\partial y} (z) = - \frac{z}{y^2} \varphi'\(\frac{z}{y}\) \in \mathcal{S}_0(0,\infty)$ since $\mathcal{S}_0(0,\infty)$ is closed under differentiation and polynomial multiplication. 
Note that the linearity of the functional $\mathcal{I}$ shows that for any $h > 0$, we have
\begin{equation*}
    \frac{\mathcal{I}\(\varphi_{y+h} \) - \mathcal{I}\(\varphi_{y} \) }{h} = \mathcal{I} \(\frac{\varphi_{y+h}-\varphi_y}{h} \).
\end{equation*}
Taking $h \rightarrow 0$ on both sides and using the continuity of $\mathcal{I}$, we have
\begin{equation*}
    \frac{\partial \mathcal{I}(\varphi_y)}{\partial y} = \mathcal{I}\(\frac{\partial \varphi_y}{\partial y} \).
\end{equation*}
With this fact, the left-hand side of Equation \eqref{eq:exchange diff and int} can be computed as
\begin{align}
    \frac{\partial }{\partial y} \int_0^\infty I(yz) \varphi(z)\d z 
    &= \frac{\partial }{\partial y} \(\frac{1}{y} \int_0^\infty I(u) \varphi\(\frac{u}{y} \)\,\d u \)\nonumber\\
    &= -\frac{1}{y^2}\int_0^\infty I(u)\(\varphi\(\frac{u}{y}\)+\frac{u}{y}\varphi'\(\frac{u}{y}\) \)\,\d u\nonumber\\
    &= -\frac{1}{y^2} \left\langle I, u \mapsto \varphi\(\frac{u}{y}\) + \frac{u}{y} \varphi'\(\frac{u}{y}\) \right\rangle.\nonumber
\end{align}
Hence, Equation \eqref{eq:exchange diff and int} holds. Finally, substituting $\varphi = \psi$ in Equation \eqref{eq:exchange diff and int} yields
\begin{equation}\label{eq:partial derivative of g}
    \frac{\partial g(t,y)}{\partial y} = \frac{\partial}{\partial y} \E\[Z_{t,T}I\(yZ_{t,T}\) \] = \frac{\partial}{\partial y} \int_0^{\infty} I(yz)\psi(z)\, \d z = \int_0^\infty I'(yz) z \psi(z)\, \d z.
\end{equation}
Let $\range (U^{**})'$ denote the range of $(U^{**})'$. Then by the definition of $I$, we have $I'(y) = 0$ for $y \in \D_I$ and $y \notin \range (U^{**})'$. Hence, we further have
\begin{equation}\label{eq:g distributional derivative}
    \frac{\partial g(t,y)}{\partial y} = \int_0^\infty I'\(yz\)z\psi\(z\) \id_{\left\{yz \in \range \(U^{**}\)' \right\}}\,\d z.
\end{equation}
Now we consider the expectation $\E\[U^{**}\(I\(yZ_{t,T} \) \) \]$. By a similar reasoning as above, we have 
\begin{equation*} 
    \frac{\partial}{\partial y} \E\[U^{**}\(I\(yZ_{t,T} \) \) \] = \int_0^\infty \(U^{**}\circ I\)'(yz) \psi(z)\, \d z,
\end{equation*}
where $(U^{**}\circ I)'$ denotes the distributional derivative of $U^{**}\circ I$. 
Note that by definition, the function $U^{**} \circ I$ has a jump at some $a_0 >0$ if and only if $I$ has a jump at $a_0$. Moreover, $\(U^{**}\circ I\)'(y) = 0$ if $y \notin \range \(U^{**}\)'$.
Since $U^{**}$ is right-hand differentiable everywhere, we can apply the chain rule to get
\begin{align*}
    \(U^{**}\circ I \)'(y) &= \(U^{**}\)'(I(y)) \cdot I'(y)
    =\left\{
    \begin{aligned}
    &y \cdot I'(y), && y \in \range \(U^{**}\)'; \\
    &0, && \text{otherwise.} 
    \end{aligned}
    \right.
\end{align*}
Then it follows that
\begin{equation*}
    \frac{\partial}{\partial y}\E\[U^{**}\(I\(yZ_{t,T}\) \) \] = y \int_0^\infty I'(yz)z\psi(z)\id_{\left\{yz\in\range \(U^{**}\)' \right\}}\, \d z.
\end{equation*}
Combining with Equation \eqref{eq:g distributional derivative}, we get
\begin{equation*}
    \frac{\partial}{\partial y} \E\[U^{**}\(I\(yZ_{t,T}\) \) \] = y\frac{\partial g(t,y)}{\partial y},
\end{equation*}
which completes the proof.
\end{proof}

We proceed to compute the explicit form of $I'$ for later contexts. Let $\{a_j\}_{j\in \mathcal{J}}$ denote the set of jump points of $I$, where the index set $\mathcal{J}$ is countable. At each $a_j$, we can compute the distributional derivative of $I$. Specifically, let $I|_{\text{ctn}}$ be the restriction of $I$ on its continuous domain. It is by the definition of $I$ that for $y \notin \range \(U^{**}\)'$, $y$ is on the continuous domain of $I$ (i.e., $I$ is continuous at $y$). Then the distributional derivative of $I$ is given by
\begin{equation}\label{eq:I distributional derivative}
    I'(y) = \(I|_{\text{ctn}}\)'(y) + \sum_{j \in \mathcal{J}} \(I\(a_j^+\)-I\(a_j^-\) \)\delta\(y - a_j\), \quad y \in \D_I,
\end{equation}
where $(I|_{\text{ctn}})'(\cdot)$ denotes the right-hand derivative of $I|_{\text{ctn}}(\cdot)$ and $\delta(\cdot)$ is the Dirac distribution. 
Note that the chain rule holds for distributional derivatives.
From Equation \eqref{eq:partial derivative of g}, the partial derivative of $g(t, y)$ can be computed as
\begin{equation}\label{eq:distributional derivative}
\begin{aligned}
   \frac{\partial g(t,y)}{\partial y} &= 
   \frac{\partial}{\partial y} \E\[Z_{t,T} I\(yZ_{t,T}\) \]\\
   &= \E\[Z_{t,T}^2\(I|_{\text{ctn}}\)'\(yZ_{t,T} \) \] \\
   &\quad+ \int_0^{\infty} z^2 \sum_{j \in \mathcal{J}} \(I\(a_j^+\)-I\(a_j^-\) \)\delta\(yz-a_j \) f_{Z_{t,T}}(z)\, \d z\\
   &= \E\[Z_{t,T}^2\(I|_{\text{ctn}}\)'\(yZ_{t,T} \)\id_{\left\{yZ_{t,T}\in \range \(U^{**}\)' \right\}} \]\\
   &\quad+\sum_{j\in \mathcal{J}}\(I\(a_j^+\)-I\(a_j^-\) \) \frac{a_j^2}{y^3}\cdot f_{Z_{t,T}}\(\frac{a_j}{y}\) .
\end{aligned}
\end{equation}
Moreover, we have the following lemma.
\begin{lemma}\label{lem:concavity of value function}
Suppose that Assumptions \ref{assp:utility}-\ref{assp} hold.
The value function $u(t,x)$ is strictly concave in $x$ for any fixed $t \in [0,T)$.
\end{lemma}
\begin{proof}
We first prove the case that $u(t,x)$ is (non-strictly) concave in $x$. Note that
\begin{equation*}
    u\(t,x\) = \sup_{\boldsymbol{\pi}\in \mathcal{V}[t,T]}\E\[U\(X_T^{t,x,\boldsymbol{\pi}} \) \] = \E\[U^{**}\(I\(\Y(t,x)Z_{t,T} \)\) \]= \sup_{\boldsymbol{\pi}\in \mathcal{V}[t,T]}\E\[U^{**}\(X_T^{t,x,\boldsymbol{\pi}} \) \].
\end{equation*}
Now, denote by $\{X_s^{t, x, \boldsymbol{\pi}}\}_{t\leq s\leq T} $ the wealth process starting at $X_t=x$ and controlled by $\boldsymbol{\pi}\in \mathcal{V}[t,T]$. Define the function
\begin{equation*}
    J\(t,x; \boldsymbol{\pi}\) := \E \[ U \( X_T^{t, x, \boldsymbol{\pi}} \) \], \quad \(t,x\) \in [0,T) \times \D_U.
\end{equation*}
Consider $x_1, x_2 \in \D_U$ and $0 < \lambda < 1$. Let $x_\lambda = \lambda x_1 + (1-\lambda) x_2$. Then $x_\lambda \in \D_U$ since $\D_U$ is a convex set. Since $u(t,x) = \sup_{\boldsymbol{\pi} \in \mathcal{V}[t,T]}\E[U(X_{T}^{t,x,\boldsymbol{\pi}})]$ for every $\epsilon >0$ and $x \in \D_U$, there exists $\boldsymbol{\pi}^\epsilon \in \mathcal{V}[t,T]$ such that
\begin{equation*}
    J\(t, x; \boldsymbol{\pi}^\epsilon\) > u\(t, x\) - \epsilon.
\end{equation*}
Hence, for a fixed $\epsilon > 0$, there exist $\boldsymbol{\pi}^{\epsilon,1}$ and $\boldsymbol{\pi}^{\epsilon,2}$ satisfying
\begin{equation*}
    J\(t, x_1; \boldsymbol{\pi}^{\epsilon,1}\) > u\(t, x_1\) - \epsilon,\quad
    J\(t, x_2; \boldsymbol{\pi}^{\epsilon,2}\) > u\(t, x_2\) - \epsilon.
\end{equation*}
Let $\boldsymbol{\pi}^{\lambda} = \lambda \boldsymbol{\pi}^{\epsilon,1} + \(1-\lambda\)\boldsymbol{\pi}^{\epsilon,2} $. It is clear that $\boldsymbol{\pi}^\lambda \in \mathcal{V}[t,T]$. From Equation \eqref{eq:SDE of wealth}, we see that the wealth process satisfies a linear SDE, which implies that the terminal wealth $X_T$ is a linear functional of the control process $\{\boldsymbol{\pi}_t\}_{0\leq t\leq T}$. Since a linear SDE admits a unique solution, the following equality holds: 
\begin{equation*}
    X_T^{t,x_\lambda, \boldsymbol{\pi}^\lambda} = \lambda X_T^{t,x_1,\boldsymbol{\pi}^{\epsilon,1}} + \(1-\lambda\)X_T^{t,x_2,\boldsymbol{\pi}^{\epsilon,2}}.
\end{equation*}
Then we have 
\begin{align*}
    u\(t, x_\lambda\) = \sup_{\boldsymbol{\pi}\in \mathcal{V}[t,T]}\E\[U^{**}\(X_T^{t,x_\lambda, \boldsymbol{\pi}} \)\]
    &\geq \E\[U^{**}\(X_T^{t,x_\lambda,\boldsymbol{\pi}^\lambda } \) \]\\
    &= \E\[U^{**}\(\lambda X_T^{t,x_1,\boldsymbol{\pi}^{\epsilon,1} }+\(1-\lambda\) X_T^{t,x_2,\boldsymbol{\pi}^{\epsilon,2} } \) \]\\
    &\geq \lambda \E\[U^{**}\(X_T^{t,x_1,\boldsymbol{\pi}^{\epsilon,1}} \) \] +\(1-\lambda\)\E\[U^{**}\(X_T^{t,x_2,\boldsymbol{\pi}^{\epsilon,2} } \) \]\\
    &\geq \lambda J\(t, x_1;\boldsymbol{\pi}^{\epsilon,1} \) + \(1-\lambda\)J\(t, x_2; \boldsymbol{\pi}^{\epsilon,2} \)\\
    &> \lambda u\(t, x_1\)+ \(1-\lambda\)u\(t,x_2\)-\epsilon.
\end{align*}
Taking $\epsilon \downarrow 0$, we have that $u(t,x)$ is concave in $x$.

Next, we rely on the following lemma to prove that $u(t,x)$ is strictly concave in $x$. The proof is given in the appendix.
\begin{lemma}\label{lem:HJB equation}
Suppose that Assumptions \ref{assp:utility}-\ref{assp} hold.
For $(t,x)\in [0,T)\times\D_U$, $\varphi \in C^{1,2}([0,T)\times \D_U)$, and $
\boldsymbol{\pi} \in \R^d$, define the operator
\begin{align*}
    \L \(t,x,\varphi;\boldsymbol{\pi}\) &:= \( r\(x - \sum_{i=1}^d \pi_{i}\) +  \sum_{i=1}^d  \mu_i \pi_{i}\) \frac{\partial \varphi(t,x)}{\partial x}\\
    &\quad \quad + \frac12 \(\sum_{i=1}^d \pi_{i}\boldsymbol{\sigma}_i\)^\top\(\sum_{i=1}^d \pi_{i}\boldsymbol{\sigma}_i\) \frac{\partial^2 \varphi (t,x)}{\partial x^2}.
\end{align*}
    Assume that $u(t,x)$ is finite and continuous for all $(t,x) \in [0,T) \times \D_U$. Then it is a viscosity solution of the following HJB equation:
    \begin{equation}\label{eq:HJB of u}
    -\frac{\partial \varphi(t,x)}{\partial t} - \sup_{\boldsymbol{\pi} \in \R^d}\L\(t,x,\varphi;\boldsymbol{\pi}\) = 0, \quad (t,x) \in [0,T) \times \D_U.
\end{equation}
\end{lemma}

Suppose that there exist $a, b \in \R$ and $t\in[0,T)$ such that $u(t,\cdot)$ is linear on $(a,b)$. It follows that $\frac{\partial^2 u(t,x)}{\partial x^2} = 0$ for $x \in (a,b)$. 

We first show that $u(t,x)$ is strictly increasing in $x$ for fixed $t \in [0,T)$. The definition of $u(t,x)$ and the attainability of the solution to Problem \eqref{prob-kernel} show that there exists $\bar{\boldsymbol{\pi}} \in \mathcal{V}[t,T]$ such that $u(t,x) = J(t,x; \bar{\boldsymbol{\pi}})$. From the definition of $U$ and Equation \eqref{eq:SDE of wealth}, for two initial values $x_1 < x_2$, we have $J(t,x_1; \bar{\boldsymbol{\pi}}) < J(t,x_2; \bar{\boldsymbol{\pi}})$. This shows that for any $t\in[0,T)$, $u(t,\cdot)$ is strictly increasing in $x$.


Since $u(t,x)$ is strictly increasing for any $t\in[0,T)$, $\frac{\partial u(t,x)}{\partial x}$ is a positive constant and $\frac{\partial^2 u(t,x)}{\partial x^2} = 0$ for $x\in(a,b)$. Utilizing the distributional derivative, the differentiability of $u(t,x)$ in $t$ can be proved by a similar argument in the proof of Lemma \ref{lem:distributional derivative}.
From Equations \eqref{eq:form of value function} and \eqref{eq:I distributional derivative}, we have
\begin{align}
        \frac{\partial u(t,x)}{\partial t} &= \frac{\partial}{\partial t}\E\[U^{**}\(I\(\Y(t,x)Z_{t,T} \) \) \]\nonumber\\
        &= \frac{\partial }{\partial t}\int_0^{\infty} U^{**}\(I\(\Y(t,x)z \) \)f_{Z_{t,T}}(z) \, \d z\nonumber\\
        &= \int_0^{\infty} z^2\Y(t,x)\frac{\partial \Y(t,x)}{\partial t}\(I|_{\text{ctn}}\)'\(\Y(t,x)z\)f_{Z_{t,T}}(z)\,\d z \nonumber\\
        &\quad +\frac{\partial \Y(t,x)}{\partial t} \sum_{j\in\mathcal{J}}\(I(a_j^+)-I(a_j^-) \)\frac{a_j^2}{\Y(t,x)^2}f_{Z_{t,T}}\(\frac{a_j}{\Y(t,x)} \)\nonumber\\
        &\quad+\int_0^{\infty}  U^{**}\(I\(\Y(t,x)z \) \) \frac{\partial f_{Z_{t,T}}(z)}{\partial t}\, \d z,\nonumber
\end{align}
which is continuous in both $t$ and $x$ since each component is continuous in $t$ and $x$.
Further, Proposition 4.5.6 in \cite{YZ1999} gives an equivalent definition of viscosity solution, which utilizes the \textit{second-order parabolic (super-)subdifferential}. In the case where $u(t,x)$ is differentiable in $t$ and twice differentiable in $x$ at some point $(\tilde{t}, \tilde{x})$, the second-order parabolic (super-)subdifferential simply consists of the classic derivatives of $u(t,x)$. Hence, the proposition shows that $u(\tilde{t},\tilde{x})$ satisfies the following equation:
\begin{equation}\label{eq:classic HJB equation}
\begin{aligned}
    &\left.\frac{\partial u(t,x)}{\partial t}\right|_{(t,x)=(\tilde{t}, \tilde{x})}  + \sup_{\boldsymbol{\pi}_t \in \R^d}\left\{\( r\(x - \sum_{i=1}^d \pi_{i,t}\) +  \sum_{i=1}^d  \mu_i \pi_{i,t}\) \left.\frac{\partial u(t,x)}{\partial x}\right|_{(t,x)=(\tilde{t}, \tilde{x})}\right.\\
    &\left.+ \frac12 \(\sum_{i=1}^d \pi_{i,t}\boldsymbol{\sigma}_i\)^\top\(\sum_{i=1}^d \pi_{i,t}\boldsymbol{\sigma}_i\) \left.\frac{\partial^2 u (t,x)}{\partial x^2}\right|_{(t,x)=(\tilde{t}, \tilde{x})}  \right\} = 0.
\end{aligned}
\end{equation}
We have shown that $u(t,x)$ is continuously differentiable in $t$. For any $(\bar{t}, \bar{x})$ such that $u(\bar{t}, \cdot)$ is linear in a neighborhood of $\bar{x}$, the second-order partial derivative term in the Hamiltonian in Equation \eqref{eq:classic HJB equation} vanishes. Thus, the following equation holds:
\begin{equation*}
\left.\frac{\partial u(t,x)}{\partial t}\right|_{(t,x)=(\bar{t}, \bar{x})} + \sup_{\boldsymbol{\pi}_t \in \R^d} \left\{ \( r\(x - \sum_{i=1}^d \pi_{i,t}\) +\sum_{i=1}^d \mu_i \pi_{i,t} \)\cdot \left.\frac{\partial u(t,x)}{\partial x}\right|_{(t,x)=(\bar{t}, \bar{x})} \right\} = 0.
\end{equation*}
Recall that $\frac{\partial u(t,x)}{\partial x} |_{(t,x) = (\bar{t}, \bar{x})} > 0$. Since $\{\boldsymbol{\pi}_t\}_{0\leq t\leq T}$ can be unbounded, the Hamiltonian is unbounded, i.e., the supremum term equals $\infty$. For a function to be the viscosity solution of the HJB equation, it is supposed to be continuous and hence finite. This leads to a contradiction, which concludes the proof.
\end{proof}

Now, with Lemmas \ref{lem:distributional derivative}-\ref{lem:HJB equation}, we are ready to present the proof of Theorem \ref{prop of V}.
\begin{proof}[Proof of Theorem \ref{prop of V}]
		From Lemma \ref{lem:form of I}, we have
            \begin{equation*}
            V\(y\) = \sup_{x\in \D_U} \{U\(x\) -xy\} = U\( I(y)\) - y I\(y\),\quad y \in \D_I.
            \end{equation*}
            From Lemma \ref{lem:concavity of value function}, we know that for a given $y >0$, $u(t,x) - xy$ is strictly concave in $x$. Hence, for each $y\in \D_I$, we have that $i(t,y)$ is the unique solution of
            $$
            \left.\frac{\partial u ( t,x) }{\partial x}\right|_{x =  i(t,y)} = y.
            $$
            From Lemma \ref{lem:distributional derivative} and Equation \eqref{eq:I distributional derivative}, we see that the partial derivative of the value function can be computed as
            \begin{equation}\label{eq:equality of derivatives}
            \begin{aligned}
              \frac{\partial}{\partial x}  \E\[U^{**}\(I\(\Y(t,x)Z_{t,T} \) \) \] &=\Y(t,x)\frac{\partial \Y(t,x)}{\partial x}\cdot \Bigg\{ \E[Z_{t,T}^2 \(I|_{\text{ctn}}\)'\(\Y(t,x) Z_{t,T}\)]\\
              &\quad \left. + \int_0^{\infty} z^2 \sum_{j \in J} \(I\(a_j^+\)-I\(a_j^-\) \)\delta\(\Y(t,x) z-a_j \) f_{Z_{t,T}}(z)\, \d z\right\}\\
              &=\Y(t,x)\frac{\partial \Y(t,x)}{\partial x} \left.\frac{\partial g(t,y)}{\partial y}\right|_{y = \Y(t,x)},
            \end{aligned}
            \end{equation}
            which is equivalent to
            \begin{equation}\label{eq:deriv of u}
                \frac{\partial u(t,x)}{\partial x} = \Y(t,x)\frac{\partial \Y(t,x)}{\partial x}\left.\frac{\partial g(t,y) }{\partial y}\right|_{y = \Y(t,x)}.
            \end{equation}
            Moreover, by definitions of $g(\cdot,\cdot)$ and $\Y(\cdot,\cdot)$, we have
            \begin{equation}\label{eq:deriv of Y}
            \begin{aligned}
            \frac{\partial \Y(t,x)}{\partial x}  &= \frac{1}{\left.\frac{\partial g\(t,y\)}{\partial y}\right|_{y = \Y\(t,x\)} }.
            \end{aligned}
            \end{equation}
            Again, utilizing the definition of $\Y(\cdot,\cdot)$, it is clear that
            \begin{equation}\label{eq:inverse functions}
            \Y\(t, \E[Z_{t,T}I\(yZ_{t,T}\)] \) = y.
            \end{equation}
            Combining Equations \eqref{eq:deriv of u}, \eqref{eq:deriv of Y} and \eqref{eq:inverse functions}, we get
            \begin{equation*}
            \left.\frac{\partial u ( t, x)}{\partial x}\right|_{x = g(t,y)} = y.
            \end{equation*}     
            This shows that  
            \begin{align*}
            v(t, y) &= \E[U(I(y Z_{t,T})) - y Z_{t,T} I(y Z_{t,T})] = \E[V(y Z_{t,T})],&& y \in \D_I,\\
            i(t, y) &= \E[Z_{t,T} I(y Z_{t,T})] = g(t,y),&& y \in \D_I.
            \qed
            \end{align*}
            

\end{proof}
\subsection{Proofs of Theorems \ref{thm:Lagrange multiplier identity and homogeneity}-\ref{thm-opt-value}}

\begin{proof}[Proof of Theorem \ref{thm:Lagrange multiplier identity and homogeneity}]
(i) Recall the definition of $g(t,y)$ and the fact that $\E[U\(I\(yZ_{t,T} \) \)] = \E[U^{**}\(I\(yZ_{t,T} \) \)]$.
Theorem \ref{prop of V} shows that
\begin{equation*}
    v(t,y)= \E\[U\(I\(yZ_{t,T} \) \) - yZ_{t,T}I\(yZ_{t,T}\) \] = \E\[U^{**}\(I\(yZ_{t,T} \) \)\] - y g(t,y).
\end{equation*}
For a fixed $t \in [0,T)$, we can compute the partial derivative of $v(t,y)$ with respect to $y$ by Lemma \ref{lem:distributional derivative}:
\begin{equation}\label{eq:deriv of v}
\begin{aligned}
    \frac{\partial v(t,y)}{\partial y} &= \frac{\partial}{\partial y} \E\[U^{**}\(I\(yZ_{t,T} \) \)\] - y \frac{\partial g(t,y)}{\partial y} - g(t,y)\\
    &= -g(t,y).
\end{aligned}
\end{equation}
Note that $v(t,y)$ is strictly convex in $y$ since Lemma \ref{lem:concavity of value function} shows that $u(t,x)$ is strictly concave in $x$. From Equations \eqref{def nu_t} and \eqref{eq:deriv of v}, we have that $\lambda(t,x)$ is the unique solution of
\begin{equation*}
    \left.\frac{\partial v(t,y)}{\partial y}\right|_{y = \lambda(t,x)} = -g\(t, \lambda(t,x)\) = -x,\quad x \in \D_U.
\end{equation*}
On the other hand, Equation \eqref{eq:def of Y} shows that $\Y(t,x)$ is the unique number satisfying $g(t,\Y(t,x)) = x$ for $x \in \D_U$. Hence, for any $t\in[0,T)$ and $x \in \D_U$, we have $\Y(t,x) = \lambda(t,x)$.

To show the second equality in (i), let $G\(t,y\) := u(t, g(t,y))$. Recalling the forms of $u(t,x)$ and $g(t,y)$, we immediately see that
        \begin{equation}\label{eq:derivative duality}
        \frac{\partial G\(t,y\)}{\partial y} = y \frac{\partial g\(t,y\)}{\partial y}.
        \end{equation}
        By the definition of $g(t,y)$, we know that 
        \begin{equation*}
            u \(t, x\) = G\(t, \Y\(t,x\) \).
        \end{equation*}
        Differentiating both sides in $x$ and combining with Equation \eqref{eq:derivative duality}, we have
        \begin{equation}\label{eq:deriv and multiplier}
        \begin{aligned}
            \frac{\partial u \(t,x\)}{\partial x} &=\frac{\partial \Y\(t,x\)}{\partial x}\cdot\left. \frac{\partial G\(t, y \)}{\partial y}\right|_{y = \Y\(t,x\)} \\
            &= \frac{\partial \Y\(t,x\)}{\partial x} \Y\(t,x\) \left.\frac{\partial g \(t, y \) }{\partial y}\right|_{ y = \Y\(t,x\)} 
            = \Y\(t, x\).
        \end{aligned}
        \end{equation}
Hence, we have for any $t\in[0,T)$ and $x\in \D_U$,
\begin{equation*}
  \Y(t,x) = \lambda(t,x)= \frac{\partial u(t,x)}{\partial x}.
\end{equation*}

\noindent
(ii) The first part of the proof shows that $\lambda\(t,x\) = \Y(t,x)$. Recall that in Section \ref{sec:DMCT}, the optimal wealth process is constructed as follows:
\begin{equation}\label{eq:def of X_t^*}
X_t^* = \E\[Z_{t,T}I\(\lambda\(0,x_0\)\xi_t Z_{t,T} \)|\xi_t \].
\end{equation}
Also, we have
\begin{equation*}
\E\[Z_{t,T}I\(\lambda(t,x)Z_{t,T} \) \] = x,\quad x \in \D_U.
\end{equation*}
Since $Z_{t,T}$ is independent of $\xi_t$, we can write
\begin{equation*}
    \E\[Z_{t,T}I\(\lambda(t,x)Z_{t,T} \) \] = \E\[Z_{t,T}I\(\lambda(t,x)Z_{t,T} \) | \xi_t\] = x.
\end{equation*}
Note that $X_t^*$ is $\sigma\(\xi_t\)$-measurable. Hence, we can substitute $x = X_t^*$ to get
\begin{equation}\label{eq:Lagrange multiplier equality}
      X_t^* =\E\[Z_{t,T}I\(\lambda(t,X_t^*)Z_{t,T} \) | \xi_t\].
\end{equation}
Combining Equations \eqref{eq:def of X_t^*}-\eqref{eq:Lagrange multiplier equality}, we get
\begin{equation*}
\E\[Z_{t,T}\(I\(\lambda(0,x)\xi_t Z_{t,T} \)- I\(\lambda(t,X_t^*)Z_{t,T} \) \)|\xi_t \] = 0,\quad \text{almost surely}.
\end{equation*}
Both $\lambda(0,x)\xi_t$ and $\lambda(t,X_t^*)$ are $\sigma(\xi_t)$-measurable. Hence, it suffices to show that for any two constants $a, b > 0$, if 
\begin{equation*}
\E\[Z_{t,T}\(I\(aZ_{t,T}\)- I\(b Z_{t,T}\) \)\] = 0,
\end{equation*}
then we must have $a = b$ since $I$ is decreasing and non-constant. 
Consequently, we prove that $\lambda(0,x_0)\xi_t =\lambda(t,X_t^*) $ almost surely.
\end{proof}

\begin{proof}[Proof of Theorem \ref{thm-opt-value}]

  (i) From Theorem \ref{thm:Lagrange multiplier identity and homogeneity}, we have $\lambda(0, x_0) = \Y(0,x_0)$. Then it follows that $X_T^* = I(\lambda (0,x_0)\xi_T)$. From Equation \eqref{eq:def of X_t^*}, $X_t^*$ is given by
    \begin{equation*}
    X_t^* 
    = \E \[ Z_{t,T} I\( \lambda(0,x_0) \xi_t Z_{t,T} \) \mid \xi_t \].
    \end{equation*}
    On the other hand, we have
    \begin{equation*}
    i\( t, \lambda(0,x_0) \xi_t \) = i\(t, \lambda(t, X_t^*)\) = \E \[ Z_{t,T} I\( \lambda(0,x_0)\xi_t Z_{t,T} \)|\xi_t \] = X_t^*.
    \end{equation*}\\
(ii) As shown in Section \ref{sec:DMCT}, the optimal control $\{\boldsymbol{\pi}_t^*\}_{0\leq t\leq T}$ is given by
    \begin{equation*}
    \begin{aligned}
    \boldsymbol{\pi}_t^* &= -\boldsymbol{\sigma}^{-1} \boldsymbol{\theta} \xi_t \frac{\partial X_t^*\(\xi_t\)}{\partial \xi_t}\\
    &= -\boldsymbol{\sigma}^{-1} \boldsymbol{\theta} \xi_t \lambda(0,x_0) \left.\frac{\partial i\(t,y\)}{\partial y} \right|_{y = \lambda(0, x_0)\xi_t}\\
    &= -\boldsymbol{\sigma}^{-1} \boldsymbol{\theta}\lambda(t,X_t^*)\left. \frac{\partial i\(t,y\) }{\partial y}\right|_{y = \lambda(t, X_t^*)},
    \end{aligned}
    \end{equation*}
    proving the result.
\end{proof}

\section{Numerical Example and Economic Insight}\label{sec:example}
In this section, we present a numerical example to illustrate the economic insight of our main Theorem \ref{thm:Lagrange multiplier identity and homogeneity}.

We consider the non-concave and discontinuous utility $U: [0,\infty) \rightarrow \R$ given by
\begin{equation*}
    U(x)=\left\{
    \begin{aligned}
        &0, && x\in[0,1];\\
        &\log(x)+1, && x \in(1,\infty).
    \end{aligned}
    \right.
\end{equation*}
We can directly compute the Legendre-Fenchel transform \eqref{eq:Legendre-Fenchel transform} of this utility function:
\begin{equation}\label{eq:ex Legendre transfrom}
\begin{aligned}
    V(y) &= \left\{
    \begin{aligned}
    &-\log(y), && y \in (0, 1);\\
    &0, && y \in [1, \infty).
    \end{aligned}
    \right.\quad \quad
    &&I(y) = \left\{
    \begin{aligned}
    & y^{-1} , && y \in (0, 1);\\
    & 0, && y \in [1, \infty).
    \end{aligned}
    \right. \\
    U^{**}(x) &= \left\{
    \begin{aligned}
    & x, && x \in [0, 1];\\
    & \log(x) +1, && x \in (1, \infty).
    \end{aligned}  
    \right.\quad\quad 
    &&\Lambda(x) = \left\{
    \begin{aligned}
    & 1, && x \in [0,1];\\
    & x^{-1}, && x \in (1,\infty).
    \end{aligned}   
    \right.
\end{aligned}
\end{equation}
Define the function
\begin{equation*}
    d(t,y) := \hat{d}\(t, \frac{1}{y}\) = -\frac{1}{\|\btheta\|_2\sqrt{T-t}}\(\log\(\frac{1}{y}\)+\(r+\frac{\|\btheta\|_2^2}{2}\)(T-t) \).
\end{equation*}
For fixed $t \in [0,T)$ and $x \in [0,\infty)$, we can compute the value function: 
\begin{equation}\label{eq:ex u}
\begin{aligned}
    u(t,x) &= \E\[U\(I\(\Y(t,x)Z_{t,T} \) \) \]\\
    &=\E\[U^{**}\(I\(\Y(t,x)Z_{t,T} \) \) \]\\
    &= \E\[\(1-\log\(\Y(t,x)Z_{t,T} \) \)\id_{\{\Y(t,x) Z_{t,T}<1\}} \]\\
    &= \(1 -\Phi\(d\(t, \Y(t,x)\) \) \)\(1- \log\(\Y(t,x)\)+\(r + \frac{\|\btheta\|_2^2}{2}\)\(T-t\) \)\\
    &\quad +\|\btheta\|_2\sqrt{T-t}\Phi'\(d\(t, \Y(t,x)\) \).
\end{aligned}
\end{equation}
Next, we can compute the explicit form of $g(t,y)$: 
\begin{equation}\label{eq:ex g}
    g(t,y) = \E\[Z_{t,T}I\(yZ_{t,T} \) \]= \frac{1}{y}\p\left\{yZ_{t,T} \in (0,1) \right\}= \frac{1}{y}\(1-\Phi\(d(t,y) \) \).
\end{equation}
Then the partial derivative of $g(t,y)$ with respect to $y$ is 
\begin{equation}\label{eq:ex g deriv}
    \frac{\partial g(t,y)}{\partial y} = -\frac{1}{y^2}\(1-\Phi\(d(t,y) \)+\frac{1}{\|\btheta\|_2\sqrt{T-t}}\Phi'\(d(t,y)\) \).
\end{equation}
Similarly, we can directly compute the partial derivative of $u(t,x)$ with respect to $x$:
\begin{equation}\label{eq:ex u deriv}
\begin{aligned}
    \frac{\partial u(t,x) }{\partial x} &= -\frac{\partial \Y(t,x) }{\partial x} \(\Y(t,x)\)^{-1}\left\{1-\Phi\(d\(t,\Y(t,x)\) \)+ \frac{\Phi'\(d\(t,\Y(t,x)\) \)}{\|\btheta\|_2\sqrt{T-t}}\right.\\
    &\quad \left. \times\(1-\log\(\Y(t,x)\)+\(r + \frac{\|\btheta\|_2^2}{2}\)\(T-t\) \) + d\(t,\Y(t,x)\)\Phi'\(d\(t,\Y(t,x)\) \) \right\}\\
    &= -\frac{\partial \Y(t,x) }{\partial x} \(\Y(t,x)\)^{-1}\(1 - \Phi\(d\(t,\Y(t,x)\) \) + \frac{\Phi'\(d\(t,\Y(t,x)\) \)}{\|\btheta\|_2\sqrt{T-t}}\).
\end{aligned}
\end{equation}
Recall the fact that
\begin{equation}\label{eq:ex lambda deriv}
    \frac{\partial \Y(t,x) }{\partial x} = \frac{1}{\left.\frac{\partial g(t,y)}{\partial y}\right|_{y = \Y(t,x)} }.
\end{equation}
Substituting \eqref{eq:ex lambda deriv} into \eqref{eq:ex u deriv} and combining with \eqref{eq:ex g deriv}, we get
\begin{equation*}
\frac{\partial u(t,x) }{\partial x} = \(\Y(t,x)\)^2 \cdot \(\Y(t,x)\)^{-1} = \Y(t,x),
\end{equation*}
which validates the second equality in Equation \eqref{eq:first_equal} of Theorem \ref{thm:Lagrange multiplier identity and homogeneity}. To show the first equality in Equation \eqref{eq:first_equal}, we try to find $\lambda(t,x)$ by the definition in \eqref{def nu_t}. From Theorem \ref{prop of V}, we can compute $v(t,y)$ as follows:
\begin{equation*}
\begin{aligned}
    v(t,y) &= \E\[V\(yZ_{t,T}\) \]\\
    &= -\E\[\log\(yZ_{t,T}\)\id_{\left\{yZ_{t,T}\in \(0,1\) \right\}} \]\\
    &= \(1- \Phi\(d(t,y)\) \)\(\log\(\frac{1}{y}\)+\(r+\frac{\|\btheta\|_2^2}{2} \)\(T-t\) \)+\|\btheta\|_2\sqrt{T-t}\Phi'\(d(t,y)\).
\end{aligned}
\end{equation*}
Then the partial derivative of $v(t,y)$ with respect to $y$ is given by
\begin{equation}\label{eq:ex v deriv}
\begin{aligned}
    \frac{\partial v(t,y)}{\partial y} &= -\frac{1}{y\|\btheta\|_2\sqrt{T-t}}\Phi'\(d(t,y)\)\(\log\(\frac{1}{y}\)+\(r+\frac{\|\btheta\|_2^2}{2}\(T-t\) \) \)\\
    &\quad - \frac{1}{y}\(1-\Phi\(d(t,y)\) \)-\frac{1}{y}d(t,y)\Phi'\(d(t,y)\)\\
    &=-\frac{1}{y}\(1-\Phi\(d(t,y)\) \).
\end{aligned}
\end{equation}
From \eqref{def nu_t}, we see that provided the strict convexity of $v(t,y)$ in $y$, $\lambda(t,x)$ is the unique function satisfying
\begin{equation*}
     \left.-\frac{\partial v(t,y)}{\partial y}\right|_{y=\lambda(t,x)} = x,\quad x \in [0,\infty),
\end{equation*}
i.e., for a fixed $t\in[0,T)$, $\lambda(t,x)$ is the inverse function of $-\frac{\partial v(t,y)}{\partial y}$ in $y$. Comparing \eqref{eq:ex g} and \eqref{eq:ex v deriv} and recalling the definition of $\Y(t,x)$, we get $\Y(t,x) = \lambda(t,x)$. This validates Equation \eqref{eq:first_equal} in Theorem \ref{thm:Lagrange multiplier identity and homogeneity}.  
\begin{figure}[h!]
    \centering
    \includegraphics[width=0.7\linewidth]{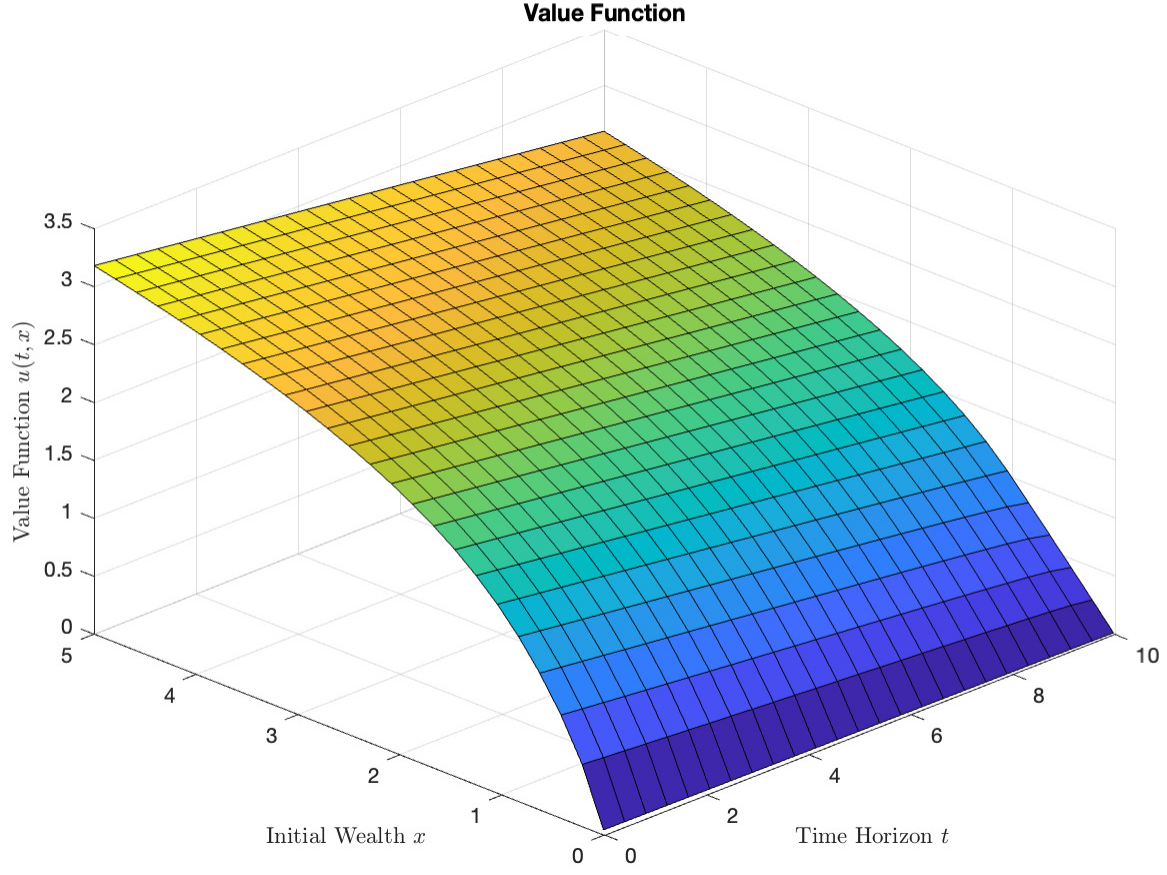}
    \caption{The value function $u(t,x)$ given by \eqref{eq:ex u}, where we let the market contain only one risky asset. The coefficients are defined as follows: $T=10$, $r = 0.05$, $\mu = 0.086$, $\sigma=0.3$.}
    \label{fig:value function}
\end{figure}
\begin{figure}[h!]
    \centering
    \includegraphics[width=0.7\linewidth]{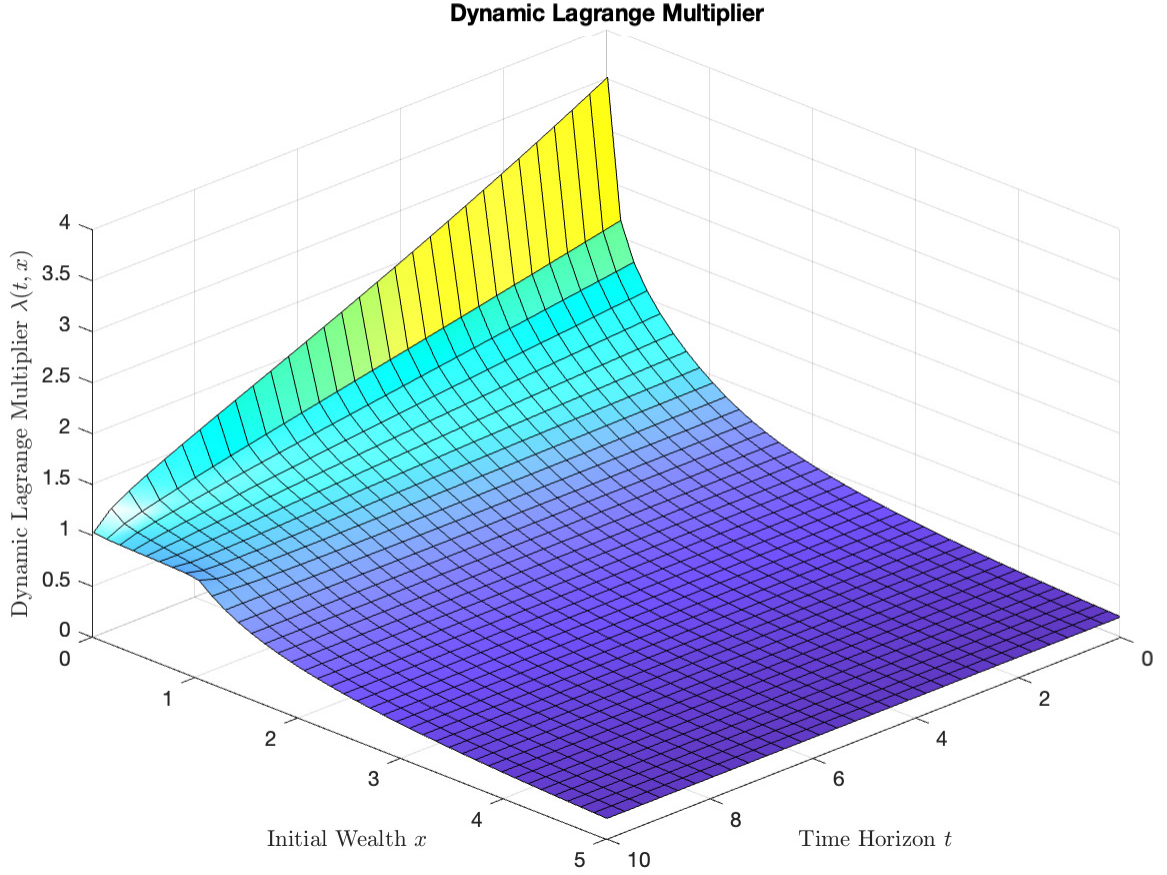}
    \caption{The dynamic Lagrange multiplier $\lambda(t,x)$, which is obtained by Legendre-Fenchel transform of $u(t,x)$ in \eqref{eq:ex u}. The coefficients are defined as follows: $T=10$, $r = 0.05$, $\mu = 0.086$, $\sigma=0.3$.}
    \label{fig:Lagrange multiplier}
\end{figure}

Finally, we present two figures to numerically illustrate the properties of $u(t,x)$ and $\lambda(t,x)$. As shown in Figure \ref{fig:value function}, for a fixed $t \in [0,T)$, the value function $u(t,x)$ is strictly concave in $x$ as shown in Lemma \ref{lem:concavity of value function}. Moreover, for a fixed $x$, the value function decreases as $t$ approaches $T$. This is reasonable since the investor has less chance to operate on his/her portfolio when it is closer to the terminal time. On the other hand, Figure \ref{fig:Lagrange multiplier} illustrates the evolution of the dynamic Lagrange multiplier $\lambda(t,x)$.
We see that when $t$ is far away from the terminal time $T$, the Lagrange multiplier $\lambda(t,x)$ is convex in $x$. This convexity fails to hold as $t$ approaches $T$, where $\lambda(t,x)$ approaches $\Lambda(x)$ in \eqref{eq:ex Legendre transfrom} simultaneously. It is because the non-concavity of the utility functions imposes more effects on the investment behaviors as it approaches the terminal time. This matches existing studies on non-concave utility maximization; see, e.g., \cite{HJ2007} and \cite{LL2020}. 
From an economic perspective, Figure \ref{fig:Lagrange multiplier} shows that for a fixed initial wealth level $x$, the shadow price $\lambda(t, x)$ decreases as $t$ approaches $T$. This reflects the diminishing marginal benefit of initial wealth (budget) and the decreasing value function over time. 

\vskip -0.2cm
\quad 

\noindent
{{\bf Acknowledgements.}
The authors are grateful to Zhichen Wang and other members of the research group on financial mathematics and risk management at CUHK-Shenzhen for their useful feedback and conversations. 
Y.L. acknowledges financial support from the National Natural Science Foundation of China (Grant No. 12401624), The Chinese University of Hong Kong (Shenzhen) University Development Fund (Grant No. UDF01003336), Guangdong Science and Technology Program (Grant No. 2024QN11X076) and Shenzhen Science and Technology Program (Grant No. RCBS20231211090814028, JCYJ20250604141203005, 2025TC0010) and is partly supported by the Guangdong Provincial Key Laboratory of Mathematical Foundations for Artificial Intelligence (Grant No. 2023B1212010001).  A.S. gratefully acknowledges support from the Natural Sciences and Engineering Research Council of Canada through grant RGPIN-2024-03761. Z.S. acknowledges financial support from the undergraduate research awards at The Chinese University of Hong Kong, Shenzhen. 
}


\newpage
{\textbf{Online 
Appendix - Auxiliary Proofs}
}
\begin{proof}[Proof of Lemma \ref{lem:enve properties}]
    (a) If there are $x<y$ such that $U^{**}(x)\geq U^{**}(y)$, then the concavity of $U^{**}$ and the monotonicity of $U$ imply that $U^{**}(z)\to-\infty$ or $U^{**}(z)\to C$ for some constant $C$ as $z\uparrow\infty$, in contradiction to the fact that $U(z)\to \infty$ as $z \to \infty$.

    (b) The inclusion $\text{\rm dom } U\subseteq \text{\rm dom } U^{**}$ follows from the fact that $U^{**}\ge U$. The converse inclusion will follow if we can show that $U^{**}(L)=-\infty$ if $U(L)=-\infty$. To this end, let us assume that $U(L)=-\infty$ but $U^{**}(L)=C>-\infty$. The upper semicontinuity of $U$ implies that there exists $\varepsilon>0$ such that $U(x)<C-1$ for all $x\in(L,L+\varepsilon]$. Thus, replacing $U^{**}$ on $[L,L+\varepsilon]$ with the straight line from $C-1$ to $U^{**}(L+\varepsilon)\ge U^{**}(L)=C$ produces a concave function that dominates $U$ and is dominated by $U^{**}$, in contradiction to the fact that $U^{**}$ is the concave envelope of $U$.

    (c) The upper semicontinuous concave function $U^{**}$ is continuous on its effective domain $\text{\rm dom } U^{**}$. Hence, $U-U^{**}$ is upper semicontinuous on $\text{\rm dom } U^{**}=\text{\rm dom } U$, and so $\{x\in \text{\rm dom } U\mid U^{**}(x)>U(x)\}$ is relatively open in $\text{\rm dom } U$. Since we always have $U^{**}(L)=U(L)$, the assertion follows. 

    (d) Suppose by way of contradiction that $U^{**}$ is not affine on $(a_k,b_k)$ with $a_k<b_k$.  Then there are $x,y\in (a_k,b_k)$ such that $x<y$ and $U^{**}((x+y)/2)>(U^{**}(x)+U^{**}(y))/2$.  Let $H$ be the affine function on $[x,y]$ that coincides with  $U^{**}$ at the endpoints $x,y$. Since  $U^{**}$ is continuous and $U$ is upper semicontinuous on $[x,y]$, the infimum of $U^{**}-U$ is attained on $[x,y]$ and strictly positive. Hence, there is $\varepsilon>0$ such that $K:=\varepsilon H+(1-\varepsilon)U^{**}$ is still strictly larger than $U$ on $[x,y]$. Let
    $\widetilde U:=K\mathbbmss{1}_{[x,y]}+U^{**}\mathbbmss{1}_{[x,y]^c}
    $. Then $\widetilde U$ is strictly smaller than $U^{**}$ and still concave since $\widetilde U'_+(x)\le (U^{**})'_+(x)$ and $\widetilde U'_-(y)\ge (U^{**})'_-(y)$. But this contradicts the fact that $U^{**}$ is the smallest concave function dominating $U$.
\end{proof}

\begin{proof}[Proof of Lemma \ref{lem:form of I}] To prove that $I$ is finite, 
    consider $y > 0$. Since $\lim_{x\uparrow\infty}(U^{**})'(x) = 0$, there exists $x_y \in \R$ such that $(U^{**})'(x) < y$ for all $x \geq x_y$. By definition, we have $I(y) \leq x_y < \infty$. Hence, $I(y)<\infty$. On the other hand, we always have  $I(y)\ge L>-\infty$.

    Now we prove \eqref{eq:I is maximizer}. 
   Since the function $U^{**}$ is concave, $I(\cdot)$ satisfies
    \begin{equation*}
        U^{**}\(I(y)\) - yI(y) = \sup_{x\in \dom U^{**}}\{U^{**}(x) - xy \} = \sup_{x\in \dom U}\{U^{**}(x) - xy \},
    \end{equation*}
    where the second equality comes from Lemma \ref{lem:enve properties}(b). Part (c) of that lemma states furthermore that the set $S=\{x\mid U^{**}(x) > U(x)\}$  can be represented as the union of a sequence $\{(a_k, b_k)\}_{k=1}^\infty$ of disjoint open intervals.
    Consider $y > 0$. If $y = (U^{**})'(x_{k_0})$ for some $x_{k_0} \in (a_{k_0}, b_{k_0})$, then we have $I(y) = a_{k_0} \notin S$, since $U^{**}$ is affine on $ (a_{k_0}, b_{k_0})$ by Lemma \ref{lem:enve properties} (d). On the other hand, it is clear that $I(y) \notin S$ if $y \notin (U^{**})'(S)$.
    This shows that $I(y) \notin S$, and so $U^{**}(I(y))=U(I(y))$ for all $y > 0$. Hence, we have
    \begin{equation*}
        U\(I(y)\) - yI(y) = U^{**}\(I(y)\) - yI(y) = \sup_{x\in \dom U}\{U^{**}(x) - xy \} \geq \sup_{x\in \dom U}\{U(x) - xy \}.
    \end{equation*}
    It remains to show that $I(y) \in \dom U$. This is trivial if $\dom U$ is closed. If $\dom U$ is open,  then the upper semicontinuity of $U$ and the fact that $L$ is finite imply that $U(x)\to-\infty$ as $x\downarrow L$. This yields  $U^{**}(x)\to-\infty$ and in turn $(U^{**})'(x)\to \infty$ as $x\downarrow L$. Hence, $I(y) > L$ and thus  $I(y)\in \dom U $ for all $y > 0$. This completes the proof. 
\end{proof}
\begin{proof}[Proof of Lemma \ref{lem:HJB equation}]
The proof is an extension of the proof of Theorem 4.5.2 in \cite{YZ1999}.
We begin by proving the dynamic programming principle (DPP), i.e., we are going to show that for $t\leq s\leq T$ and $x \in \D_U$, 
\begin{equation*}
    u(t,x) = \sup_{\boldsymbol{\pi} \in \mathcal{V}[t,T]}\E_{t,x}\[u\(s,X_s^{t,x,\boldsymbol{\pi}|_{[t,s]} }\)\],
\end{equation*}
where $\E_{t,x}[\cdot]:=\E\[\cdot|X_t=x \]$ and $\boldsymbol{\pi}|_{[t,s]}$ denotes the control process from $t$ to $s$. Since the wealth process is Markovian, we have
\begin{align*}
    u(t,x) &= \sup_{\boldsymbol{\pi}\in \mathcal{V}[t,T]}\E_{t,x}\[U\(X_T^{t,x,\bpi }\) \]\\
    &= \sup_{\boldsymbol{\pi}\in \mathcal{V}[t,T]}\E_{t,x}\[\left.\E\[U\(X_T^{t,x,\bpi}\)\right|\F_s \] \]\\
    &= \sup_{\boldsymbol{\pi}\in \mathcal{V}[t,T]}\E_{t,x}\[\left.\E\[U\(X_T^{s,X_s^{t,x,\bpi|_{[t,s]}} ,\bpi|_{[s,T]}}\)\right|X_s^{t,x,\bpi|_{[t,s]}} \] \]\\
    &\geq \sup_{\boldsymbol{\pi}|_{[t,s]}\in\mathcal{V}[t,s]}\E_{t,x}\[\sup_{\boldsymbol{\pi}|_{[s,T]} \in \mathcal{V}[s,T] } \E\[\left.U\(X_T^{s,X_s^{t,x,\bpi|_{[t,s]}},\bpi|_{[s,T]}}\)\right|X_s^{t,x,\bpi|_{[t,s]}}\]\]\\
    &=\sup_{\boldsymbol{\pi}|_{[t,s]}\in\mathcal{V}[t,s]}\E_{t,x}\[u\(s,X_s^{t,x,\bpi|_{[t,s]}}\) \]\\
    &=\sup_{\boldsymbol{\pi}\in\mathcal{V}[t,T]}\E_{t,x}\[u\(s,X_s^{t,x,\bpi}\) \].
\end{align*}
On the other hand, for an arbitrary control $\boldsymbol{\pi}'\in\mathcal{V}[t,T]$, we have 
\begin{equation*}
    u\(s,X_s^{t,x,\bpi'|_{[t,s]}}\) \geq \E\[\left.U\(X_T^{s,X_s^{t,x,\bpi|_{[t,s]}},\bpi'|_{[s,T]}}\)\right|X_s^{t,x,\bpi'|_{[t,s]}} \].
\end{equation*}
Taking expectation on both sides:
\begin{align*}
    \E_{t,x}\[u\(s,X_s^{t,x,\bpi'|_{[t,s]}} \) \] &\geq \E_{t,x}\[\left.\E\[U\(X_T^{s,X_s^{t,x,\bpi|_{[t,s]}},\bpi'|_{[s,T]}} \)\right|X_s^{t,x,\bpi'|_{[t,s]}} \] \]
    = \E_{t,x}\[U\(X_T^{t,x,\bpi'} \) \].
\end{align*}
Taking supremum on both sides, we get
\begin{equation*}
    \sup_{\bpi\in\mathcal{V}[t,T]}\E_{t,x}\[u\(s,X_s^{t,x,\bpi|_{[t,s]}} \)\]\geq \sup_{\bpi \in \mathcal{V}[t,T]}\E_{t,x}\[U\(X_T^{t,x,\bpi}\) \] = u(t,x).
\end{equation*}

Now, we show that $u(t,x)$ is a viscosity supersolution. Let $\varphi \in C^{1,2}\([0,T)\times \D_U\)$ be such that $u-\varphi$ attains a local minimum at $(t',x') \in [0,T)\times \D_U$ and $u(t',x') = \varphi(t',x')$.
By the definition of local minimum, there exist $h>0$ and $\epsilon>0$ such that $t'+h<T$ and 
$u(t,x) \geq \varphi(t,x)$ for all $(t,x) \in [t',t'+h]\times B_\epsilon(x')$, where $B_\epsilon\(x'\):= \{x\in\D_U: |x-x'|<\epsilon \}$. Fixing a constant control $\bpi' \in \R^d$, let $\{X_t^{t',x',\bpi'}\}_{t'\leq t\leq T}$ be the state trajectory under the constant control.
Then from DPP and the linearity of the SDE of $\{X_t^{t',x',\bpi'}\}_{t'\leq t\leq T}$, there exists $h >0$ small enough such that
\begin{align*}
    0&\geq \frac{1}{h}\E_{t',x'}\[\(u\(t',x'\)-\varphi\(t',x'\) \) - \(u\(t'+h, X_{t'+h}^{t',x',\bpi'}\)-\varphi\(t'+h,X_{t'+h}^{t',x',\bpi'}\) \) \]\\
    &\geq \frac{1}{h}\E_{t',x'}\[\varphi\(t'+h,X_{t'+h}^{t',x',\bpi'}\)-\varphi\(t',x'\) \].
\end{align*}
By Itô's formula, taking $h\downarrow 0$, we have
\begin{equation}\label{eq:viscosity supersolution}
    -\frac{\partial \varphi\(t',x'\)}{\partial t} - \L\(t',x',\varphi; \bpi'\) \geq 0.
\end{equation}
Taking $\sup_{\bpi' \in \R_d}$ on $\L\(t',x',\varphi; \bpi'\)$, we get that $u$ is a viscosity supersolution.

On the other hand, suppose that $u-\varphi$ attains a local maximum at $(t',x') \in [0,T) \times \D_U$. From DPP, for any $\epsilon >0$ and $h > 0$ small enough, there exists $\bpi^{\epsilon, h} \in \mathcal{V}[t',T]$ such that 
\begin{align*}
    0&\leq \E_{t',x'}\[\(u\(t',x'\)-\varphi\(t',x'\) \) - \(u\(t'+h, X_{t'+h}^{t',x',\bpi^{\epsilon,h}}\)-\varphi\(t'+h,X_{t'+h}^{t',x',\bpi^{\epsilon,h}}\)\) \]\\
    &\leq \epsilon h + \E\[\varphi\(t'+h,X_{t'+h}^{t',x',\bpi^{\epsilon,h}}\) - \varphi\(t',x'\)  \].
\end{align*}
Dividing both sides by $h$, we get
\begin{align*}
    -\epsilon \leq \frac{1}{h}\E\[\varphi\(t'+h,X_{t'+h}^{t',x',\bpi^{\epsilon,h}}\) - \varphi\(t',x'\)  \].
\end{align*}
Taking $h\downarrow 0$ and since $\epsilon$ is arbitrary, we get
\begin{equation}\label{eq:viscosity subsolution}
    -\frac{\partial \varphi\(t',x'\)}{\partial t} - \L\(t',x',\varphi; \bpi'\) \leq 0.
\end{equation}
Combining Equations \eqref{eq:viscosity supersolution}-\eqref{eq:viscosity subsolution}, we prove $u(t,x)$ is a viscosity solution of Equation \eqref{eq:classic HJB equation}.

\end{proof}
\end{document}